\documentclass[11pt]{article}
\usepackage[intlimits]{amsmath}
\usepackage{amsfonts, amssymb, amsthm}
\usepackage{esint}
\usepackage{epsfig}
\usepackage{color}
\usepackage{threeparttable} 
\usepackage[toc,page]{appendix}
\renewcommand{\o}{\omega}

\renewcommand{\div}{\,\mathrm{div}\,}

\newcommand{\g}{\gamma}

\newtheorem{thm}{Theorem}[section]
\newtheorem{prop}[thm]{Proposition}

\newtheorem{lemma}[thm]{Lemma}
\newtheorem{defn}[thm]{Definition}

\newtheorem{preremark}[thm]{Remark}

\newenvironment{remark}{\begin{preremark}\rm}{\medskip \end{preremark}}
\numberwithin{equation}{section}

\newcommand{\R}{\mathbb R}

\newcommand{\grad} {\nabla}

\newcommand{\dx} {\; \mathrm{d} x}

\def\XXint#1#2#3{{\setbox0=\hbox{$#1{#2#3}{\int}$}
       \vcenter{\hbox{$#2#3$}}\kern-.5\wd0}}

\newcommand{\meanbar}[1]{%
\setbox0 = \hbox{$#1 \int$}
\hbox to 0pt{%
\thinspace
\hskip 0.1\wd0
\raise 0.5\ht0
\hbox{%
\lower 0.5\dp0
\hbox{\rule{0.8\wd0}{2\linethickness}}
}%
\hss
}%
}

    \newcounter{myfootertablecounter}

\begin{document}
\title{On the regularity of the Non-dynamic Parabolic Fractional Obstacle Problem}

%\authorrunning{Non-dynamic Parabolic Fractional Obstacle Problem} % if too long for running head

\author{Ioannis Athanasopoulos, Luis Caffarelli, Emmanouil Milakis}
\date{}
\maketitle

%\doublespacing

\begin{abstract}
In the class of the so called non-dynamic Fractional Obstacle Problems of parabolic type, it is shown how to obtain higher regularity as well as optimal regularity of the space derivatives of the solution. Furthermore, at free boundary points of positive parabolic density, it is proven that the time derivative of the solution is H\"{o}lder continuous. Finally, at regular free boundary points, space-time regularity of the corresponding free boundary is obtained for any fraction $s\in(0,1)$.
\\



\end{abstract}
%\tableofcontents
\section{Introduction}\label{section1}

  Motivated by problems in optimal control and financial engineering, considerable literature was dedicated to obstacle problems for the L\'{e}vy process i.e.
$$\min(\partial_tu+(-\Delta)^s, u-\psi)=0  \ \ \ \ \text{in} \ \ \R^{n-1}.$$
 In this work our motivation is guided by the Signorini problem or the semi-permeable membrane model. In this case $\R^{n-1}$ is our physical boundary and represents a plane ($n=3$ case) where the elastic set is lying or a semi-permeable membrane across it a chemical substance (or heat) is flowing into $\Omega$. The set $\Omega$ is composed of a non-homogeneous anisotropic material where its evolution is governed by 
 $$\partial_tu-\frac{1}{x_n^{\gamma}} \div(x_n^{\gamma} \grad u)=0 \ \ \ \ \ \ \text{in} \ \ \Omega. $$
This induces a fractional heat equation
$$H^su:=(\partial_t-\Delta)^s u=0 \ \ \ \ \text{on} \ \ \R^{n-1}$$
where $\gamma =1-2s$, for $0<s<1$. This is an example of a Master equation $Mu=0$ (see \cite{CSmaster}) which appears in the study of continuous time random walks where the random jumps occur with random time lag. We have coined our problem, as in the $\gamma =0 $ case (see \cite{ACM1}), 
non-dynamical in order to distinguish it from the ones where the processes are governed by equations of the type
$$\partial_tu+Mu=0 \ \ \ \ \text{on} \ \ \R^{n-1} \ \ $$
(see \cite{KK74} and \cite{KMS73}).

In \S\ref{SectionPrel} a number of known properties of solutions to degenerate parabolic equations are included which will be used in the subsequent sections. The formulation of our problem (\S\ref{SectionBasic}) follows the generic scheme of parabolic obstacle problems and as such existence, uniqueness as well as basic properties can be found in \cite{ACMApprox}. Thus we concentrate on studying higher regularity. We first prove (\S\ref{SectionSpace}) H\"{o}lder continuity of space derivatives for any $0<s<1$; thus extending the result of $s=\frac{1}{2}$ case (see \cite{Athparabolic}). Then in \S\ref{SectionOptimal} via a monotonicity formula similar to the one in \cite{ACM1} we obtain the optimal regularity of space derivatives. Since, in general the time derivative can experience jumps at free boundary points, we prove that at such points of positive parabolic density with respect to the coincidence set the time derivative is continuous nearby (\S\ref{SectionTime}). The quasi-convexity introduced in \cite{ACM1} together with the continuity of the time derivative of \S\ref{SectionTime} allows us to use the elliptic free boundary theory developed in \cite{ACS} and \cite{CSS}, using the appropriate elliptic Almgren's frequency formula. Therefore we are able to obtain (\S\ref{SectionFB}) space-time smoothness of the free boundary at regular points for any $0<s<1$. This is feasible since, in our case, the diffusion term is always dominant, as opposed to the one for the L\'{e}vy processes mentioned above where the free boundary theory breaks down at $s=\frac{1}{2}$. 

%Let us point out that the quasi-convexity requires that initially there is no free boundary point present. So, in case where the obstacle is non-increasing in time or is independent of time, there will be no free boundary at all for any future time. 
%Actually, in the latter case the time derivative could be non-negative (see \cite{CFi}).   

\section{Preliminaries}\label{SectionPrel}

In this section we state certain known properties of solutions to degenerate parabolic equations which will be of use in the sequel. We consider degenerate parabolic operators of the form
\begin{equation} \label{dpo}
w(x',y)\partial _t u (x',y,t) - \div (w(x',y)A \nabla u(x',y,t)) = 0
\end{equation}
where $w(x',y)$ is an $A_2-$weight in $\R ^n$. Note that $w(x',y)= |y|^\g$ with $\g \in (-1,1)$ falls into this class. Equations of type (\ref{dpo}) appear when one pulls back the heat operator via a quasiconformal mapping from $\R^n$ into $\R^n$. They also model the diffusion of the temperature in a non-isotropic and non-homogeneous material. In such case, weight $w(x',y)$ is equal to the product $d(x',y)h(x',y)$ where $d(x',y)$ is the density of the material at point $(x',y)$ while $h(x',y)$ represents the specific heat at that point (see \cite{Gu}).  We refer the reader to \cite{CHSE1}, \cite{GW}, \cite{ISH}  and references therein for the study of solutions to parabolic degenerate equations of type (\ref{dpo})

\par Let $B_1$ be the unit ball in $\R^n$ with center at the origin. We say that a function $u(x',y,t) \in L^2((-1,0); H^1(B_1, |y|^\g))=L^2(Q_1,|y|^\g)$, where $Q_1 := B_1 \times (-1,0)$, is a solution of
\begin{equation}\label{def_1}
\div (|y|^\g \nabla u(x',y,t))-|y|^\g\partial _t u(x',y,t)=0
\end{equation}
in $Q_1$ if
\begin{align*}
\int_{Q_1} ( |y|^\g u_{x_i} \eta _{x_i} - |y|^\g u \eta _t) \ dx' dy dt =0
\end{align*}
for every $\eta \in \{ \eta \in L^2((-1,0]; H_0^1(B_1, |y|^\g)): \eta _t \in L^2((-1,0]; L^2(B_1, |y|^\g)) \} $ and $\eta(-1)= \eta(0)=0$.
\par Here we denote by $L^p(B_1, |y|^\g)),\ p \geq 1$ the Banach space of all measurable functions $f$ defined on $B_1$ such that 
$$||f||_{L^p(B_1, |y|^\g))} := \left( \int_{B_1} |f(x',y)|^p |y|^\g \ dx' dy \right) ^{1/p} < \infty .$$
Also,  $H_0^1(B_1, |y|^\g)$ and $ H^1(B_1, |y|^\g)$ denote the closure of $C_0^\infty (B_1)$ and $C^\infty (\overline{B_1})$ under the weighted Sobolev norm
$$\left( \int_{B_1} |f(x',y)|^2 |y|^\g \ dx' dy +  \int_{B_1} |\nabla f(x',y)|^2 |y|^\g \ dx' dy  \right) ^{1/2}.$$

\begin{prop}\label{prop2.4}
Let $u$ be a positive solution of (\ref{def_1}) in $Q_1$. Then there exists a positive dimensional constant $C$ such that if $Q_{R} (0) \subset Q_1,\ R>0$,
$$||u||_{L^{\infty}(Q_{R/2})} \leq C ||u||_{L^p(Q_R, |y|^\g))}$$
for every $0<p<2$.
\end{prop}

\begin{proof}
Apply Lemma 2.3 of \cite{CHSE1} for $w(x,y)=|y|^\g \in A_2$.
\end{proof}

We will also need Poincar\'e type inequalities in the context of the weighted Sobolev spaces (see \cite{ChFra}, \cite{FKS}). The next proposition is a weighted imbedding theorem.

\begin{prop} \label{weighted_imb_thm}
Let $u \in L^\infty((-T,0); L^2(B_R, |y|^\g)) \cap L^2((-T,0); H_0^1(B_R, |y|^\g))$ and $Q=B_R \times (-T,0)$, $w(B_R)= \int_{B_R} |y|^\g \ dx' dy.$ Then
\begin{align*}
\left( \frac{1}{w(B_R)} \int_{Q} |u(x',y,t)|^{2 \lambda } |y|^\g \ dx' dy dt \right) ^{\frac{1}{2 \lambda} }&\leq C \left( \frac{R}{\sqrt{T}} \right) ^{\frac{1}{\lambda} } \left( \sup_{(-T,0)} \left( \frac{1}{w(B_R)} \int_{B_R} u^2(x',y,t) |y|^\g \ dx' dy  \right) ^{\frac{1}{2}} \right) ^{1-\frac{1}{\lambda } } \\
&\ \ \ \cdot \left( \frac{1}{w(B_R)} \int_{Q} |\nabla u(x',y,t)|^2  |y|^\g \ dx' dy dt \right) ^{\frac{1}{\lambda} }
\end{align*}
where $\lambda = 2 - \frac{1}{k} >1$, $k=\frac{n}{n-1}  + \delta $ for some $\delta >0$ and $C$ is a positive constant that depends only on $\g$ and the dimension.
\end{prop}

\begin{proof}
The proof follows that of Theorem 1.2 of \cite{FKS} for $A_2$-weight $w(x',y)=|y|^\g$.
\end{proof}

The following is the classical Poincar\'e inequality for functions in $W^{1,2}$-weighted Sobolev space. For the proof, we refer the reader to \cite{FKS}.

\begin{prop} \label{Poincare}
For any function $v\in W^{1,2}(B_1,|y|^\g)$ the following inequality holds
$$\int_{\partial B_r}|v(x',y)-\bar{v}|^2|y|^\g d \sigma\leq Cr\int_{B_r}|\grad v(x',y)|^2|y|^\g dx'dy$$
where $\bar{v}=\frac{1}{\omega_{n+\g}r^{n+\g}}\int_{\partial B_r}v(x',y)|y|^\g d\sigma$ for $\omega_{n+\g}:=\omega(\partial B_1)$ and $C$ is a constant that depends only on $n$ and $\g$.
\end{prop}

Harnack inequality has been proved for solutions of (\ref{dpo}) in \cite{CHSE1}.

\begin{prop} \label{Harnack_dpo}
Let $u$ be a positive solution of (\ref{dpo}) in $Q_R:= B_R \times (-R^2,0)$ where $A$ is uniformly elliptic and $w$ is an $A_2$-weight. Then there exists $C$ which depends on the dimension, the ellipticity constants and the $A_2$-constant such that
\begin{align*}
\sup_{\tilde{Q}_{R/2}^-} u(x',y,t) \leq C \inf_{\tilde{Q}_{R/4}^+} u(x',y,t)
\end{align*}
where $\tilde{Q}_R^-:= B'_{R}\times (-R,R) \times \left(- \frac{3}{4} R^2,-\frac{1}{2} R^2 \right)$ and $\tilde{Q}_R^+:= B'_{R}\times (-R,R)  \times \left( -\frac{1}{4} R^2,0\right]$.
\end{prop}

The previous result can be generalized in the case of non-homogeneous degenerate parabolic equations. First, we recall the Sobolev inequality with weights: There exist constants $C_1>0$ and $k>1$ depending only on $A_2$-constant of the weight such that
\begin{align*}
\left( \fint_{B_r} |u|^{2k} w(x',y) \ dx'dy  \right) ^{1/2k} \leq C_1 r \left( \fint_{B_r} |\nabla u|^{2} w(x',y) \ dx'dy  \right) ^{1/2}
\end{align*}
for all $u \in H_0^1(B_r,w)$. It is known that $k>\frac{n}{n-1}$ for $n \geq 2$ (see \cite{FKS}, \cite{ISH} and the references therein).

\begin{prop} \label{Harnack_f}
Let $u$ be a non-negative solution of
$$\partial _t u  - \frac{1}{w(x',y)} \div (w(x',y)A \nabla u)=f$$
in $Q=B_1 \times (-1,0)$, where $w(x',y) \in A_2,\ A$ is uniformly elliptic and $f \in L^\infty ((-1,0);L^{k'+\delta }(B_1,w))$ for some positive $\delta $ and $k'$ such that $\frac{1}{k} + \frac{1}{k'}=1$. If $Q_R \subset Q$ then there exists a constant $C$ such that 
\begin{align*}
\sup_{Q_R^-} (u+F) \leq C \inf_{Q_R^+} (u+F)
\end{align*}
where $$F= \sup_{t \in (-1,0)} \left( \fint_{B_1} |f(x',y,t)|^{k'+\delta } w(x',y) \ dx'dy \right) ^{1/(k'+\delta )}$$
and $C$ depends only on the dimension, the ellipticity constant of $A$, the $A_2$-constant of $w$ and $\delta $.
\end{prop}

\begin{proof}
Set $v=u+F$ and follow the proof of Proposition \ref{Harnack_dpo} in \cite{CHSE1}.
\end{proof}

Throughout the text, we use the notation $(x',y)$ or $(x',x_n)$, for $x'\in \R^{n-1}$, to denote a point in $x\in\R^n$. Similarly $\R^n_+:=\{(x',x_n):\ x_n>0\}$ is the half space, $B_1^+:=\{x\in B_1:\ x_n>0\}$ is the half unit ball in $\R^n$. The various constants that will appear in the sequel may vary from formula to formula, although for simplicity we use the same letter. If we do not give any explicit dependence for a constant, we mean that it depends only on $s$ and the dimension $n$.

\section{Formulation of the problem and basic properties of solutions}\label{SectionBasic}

The obstacle problem to be studied in this work is a direct generalization (as in \cite{AC2010}) of the non-dynamic parabolic thin obstacle problem, $2^{\text{nd}}$ prototype, studied in \cite{ACM1}; that is, from the heat equation valid in a region $\Omega \subset R^n$ with Signorini type boundary conditions on part of its boundary lying on $\R^{n-1}$ to a degenerate parabolic equation valid in the same region with the relevant boundary conditions. More precisely:

Given a bounded domain  $\Omega$ in $\R^n_+:=\R^{n-1}\times \{x_n\geq 0\}$ with part of its 
boundary $\Sigma\subset \partial\Omega$ that lies on $\R^{n-1}$, a function $\psi(x',t)$ (the obstacle) 
where $\psi<0$ on $(\partial \Omega\setminus\Sigma)\times(0,T]$, $\max \psi(x',0)>0$, a 
function $\phi$ (boundary and initial data) with $\phi=0$ on $(\partial \Omega\setminus\Sigma)\times(0,T]$, $\phi\geq \psi$ on $\Sigma\times\{0\}$, and an $f$ (non-homogeneous term), find a function $u$ such that 
\begin{align}
\begin{cases} \label{FBP3}
\div (y^\g \nabla u)-y^\g\partial _t u  = f ,\ \ & \text{for} \ (x',y,t) \in \Omega \times(0,T] \\
u(x',0,t) \geq \psi (x',t), \ \ & \text{for} \ (x',t) \in \Sigma \times (0,T] \\
\lim_{y \to 0^+} y^\g  u_y (x',y,t)=0, \ \ & \text{for} \ u(x',0,t) > \psi (x',t) \\
\lim_{y \to 0^+} y^\g  u_y (x',y,t) \leq 0,\ \ & \text{for} \ (x',y,t) \in \Omega \times (0,T] \\
u(x',y,t)=\phi (x',y,t),\ \ & \text{on} \  \partial_p (\Omega \times (0,T]) \end{cases}
\end{align}
for any $\g\in(-1,1)$, where we have set $x_n=y$ and denote by $\partial_p$ the parabolic boundary.

With no loss of generality, we can assume $f=0$ since for $f$ smooth enough we can subtract from the solution as well from the obstacle a function $w$ which solves the  equation $\div (y^\g \nabla w)-y^\g\partial _t w  = f$.

Also, this problem may be stated in the whole space: Given a function $\psi(x',t)$ with $x'\in \R^{n-1}$ and $t\in \R$ decaying very rapidly at infinity, find a $u(x',y,t)$ where $y=x_n$ such that  
\begin{equation}\label{model2in}
\begin{cases}  
\div(y^\gamma \grad u)-y^\gamma \partial_t=0,   &{\rm{in}} \ \  \R^n_+ \times (0,T]  \cr
\lim_{y \to 0+}y^\gamma u_y \leq 0, \ \ u\geq \psi &{\rm{on}} \ \  \R^{n-1} \times (0,T]  \cr
(\lim_{y \to 0+}y^\gamma u_y)(u-\psi)=0  &{\rm{on}} \ \  \R^{n-1} \times (0,T]  \ \cr
u=\phi &{\rm{on}} \ \   \R^n_+\times \{0\}\cr
\end{cases}
\end{equation}
%\label{model2in}

%The fractional Laplacian of a function $u: \R ^{n-1} \to \R$ is expressed by the formula
%\begin{align*}
%(- \Delta )^s u(x')= C_{n,s} \int_{\R ^{n-1}} \frac{u(x')-u(z')}{|x'-z'|^{n-1+2s}} \ dz'
%\end{align*}
%for $s \in (0,1)$ and $C_{n,s}$ a normalization constant, where the integral is considered in the principal value sense. 

By the $L_\g (:=div(y^\g\grad)-y^\g\partial_t)$-caloric extension formula (see \cite{ST15} and \cite{NS}) in all of $\R$ (\ref{model2in}) can be written as  
\begin{equation}\label{model3in}
\begin{cases}  
u(x',t)\geq \psi(x',t) &{\rm{for}} \ \  \R^{n-1} \times \R  \cr
(\partial_t-\Delta)^s u(x',t)=0 \ \  &{\rm{for}} \ \  u(x',t)> \psi(x',t) \cr
(\partial_t-\Delta)^s u(x',t)\geq 0  &{\rm{for}} \ \  \R^{n-1} \times \R \ \cr
\end{cases}
\end{equation}
or, as a Hamilton-Jacobi equation,
\begin{equation}
{\rm{min}}\{H^su,u-\psi\}=0  
\ \ \ \  \  \rm{in} \ \  \R^{n-1}\times \R \ 
\end{equation} 
where $s=\frac{1-\g}{2}$ and $$H^su(x',t):=(\partial_t-\Delta)^su(x',t)=-c(s)\lim_{y\to 0^+}(y^{\g}u_y(x',y,t))$$ 
for some normalizing positive constant $c(s)$.

Observe that if, in (\ref{FBP3}), 
we confine ourselves to a domain contained in 
$\Omega \times (0,+\infty)$ whose closure does not meet $\Sigma \times (0,+\infty)$ then the problem is just a standard initial-boundary value problem of degenerate parabolic equation studied by \cite{CHSE1} and others. Therefore we focus our attention on and near 
$\Sigma \times (0,\infty)$; a fortiori in a neiborhood of a free boundary point. On the other hand, it is more convenient to study (\ref{model3in}) in its localized form (\ref{model2in}) near a free boundary point. Hence, in this context, the problems are identical.

As it was mentioned in the introduction existence and certain a priori bounds follow from a generic scheme for parabolic problems. For a proof of the following lemma see \cite{ACMApprox}.

\begin{lemma} \label{a_priori}
Let $\psi(x',t)$ be a $C^2$ function and $\phi(x',y,0)$ a $C^4$ one then there exists a solution $u$ of (\ref{FBP3}) or (\ref{model2in}) such that
\\ \textbf{(i)} $\grad_{x'}u$, $y^\g u_y$,  $u_t$   are $L^\infty$ bounded depending on $\psi$, the initial data, and $||u||_{L^\infty}$
\\ \textbf{(ii)} $u$ is semi-convex i.e. $||(u_{\tau\tau})^-||_{L^\infty}<C(\psi, \phi|_{t=0},||u||_{L^{\infty}})$ where $\tau$ a unit vector in $\R^{n-1}$
\\ \textbf{(iii)} $u$ is quasi-convex   i.e. $||(u_{\zeta\zeta})^-||_{L^\infty}<C(\psi,||u||_{L^\infty})$ where $\zeta$ a unit vector in $(x',t)$ variables provided $(u-\psi)|_{t=0}>0$. 

%\\ Furthermore, in case $(u-\psi)|_{t=0}\geq 0$, $(iii)$ is still valid if, in addition, $\Delta (\partial_y\phi)|_{t=0}=o(|y|)$ and $\Delta_x'\psi+\partial_t\psi\geq M$ for $M$ sufficiently large in space.
\end{lemma}

\begin{remark} \label{quasi}
In $(iii)$ of the above lemma we use the term quasi-convexity instead of semi-convexity. The main reason, for this, is to emphasize the qualitative difference stated in $(ii)$ and $(iii)$. Also, to point out that this important and simple property common to "all" obstacle problems of parabolic type, absent in the literature, was first introduced in \cite{ACM1} (its results were presented at an international conference held in Brazil, August 2015, see \cite{ACM1}).
\end{remark}

\begin{remark}\label{HeatDeg} 
In \cite{ST15} (see also \cite{NS}) the fractional heat operator has been realized as a parabolic hypersingular integral and fine properties of solutions have been derived through its connection to the corresponding extension problem which involves a parabolic degenerate equations with an $A_2$ weight. In \cite{AC2010} this degenerate equation has been considered in the study of the continuity of the temperature in boundary heat control problems with fractional diffusion. 
\end{remark}

\section{H\"{o}lder continuity of the space derivative}\label{SectionSpace}

In the present section we show the H\"{o}lder continuity of the space derivative. Let us point out first what we know already about our solution. 

\begin{lemma}\label{lemma3.1}
Let $u$ be a solution to (\ref{FBP3}) in $Q_1^+:=B_1^+\times (-1,0)$. Then

(a) $u(x',y,t)+\frac{c_0}{2}|x'|^2$ is convex in $x'$ where $c_0:=||u_{\tau\tau}||_\infty$.

(b) $\partial_y(y^\gamma u_y(x',y,t))\leq nC_0y^\gamma$ where $$C_0:=||u_{\tau\tau}^-||_\infty+||\psi_{\tau\tau}||_\infty+||u_{t}||_\infty+||\psi_{t}||_\infty.$$

(c) $u(x',y,t)-u(x',0,t)\leq \frac{nC_0}{2(1+\gamma)}y^\gamma$ for all $x'$ and $t$.

(d) If $(u-\psi)(x',y,t)\geq h$ then $(u-\psi)(z',y,s)\geq h-C_0\rho^2$ in the half cylinder $$HQ_\rho'(x',t):=\left\{(z',s)\in Q'_\rho(x',0,t): \grad_x'u(x',y,t)\cdot(z'-x')\geq 0\right\}.$$ 

\end{lemma}

\begin{proof}
It is an immediate consequence of Lemma \ref{a_priori}.
\end{proof}

 Set $w(x',y,t):=y^\gamma u_y(x'y,t)$. Since $\partial_y(y^\gamma u_y(x',y,t))\leq nC_0y^\gamma$ 
the following limit is well defined i.e $$\lim_{y\to0^+}w(x',y,t)=:w(x',0,t)\leq 0.$$

%Next we state and prove the main theorem of the present section.

\begin{thm}\label{thm3.1}
Under the above assumptions, near a free boundary point, $w$ is H\"{o}lder continuous whose exponent depends on $L^\infty$, the Lipschitz constant of the solution, the obstacle $\psi$ and on $C_0$. 
\end{thm}
\begin{proof}
Assuming $(0,0,0)$ a free boundary point, it is enough to prove 
\begin{equation}\label{eq3.0}
\inf_{\Gamma_r}w\geq -Cr^\alpha
\end{equation}
where $\Gamma_r:=Q'_r\times[0,\frac{\sqrt{1+\gamma}}{2n}r)$.  In fact we prove that there exists $C>0$ and $0<\mu<1$ depending only on the $L^\infty$ bound of the solution and the obstacle, on the $L^\infty$  bound of their space derivatives and on $C_0$ such that
\begin{equation}\label{eq3.1}
\inf_{\Gamma_{4^{-k}}} w\geq -C\mu^k, \ \ \ \forall k\in \mathbb{N}.
\end{equation}
By induction we assume that (\ref{eq3.1}) is true for every $k\leq k_0$ i. e
$$\inf_{\Gamma_{4^{-k_0}}} w\geq -C\mu^{k_0}$$
for some constant $C>0$ and $0<\mu<1$ to be chosen. Normalize the solution to $\Gamma_1$ by setting 
$$\overline{u}(x',y,t):=C^{-1}\left(\frac{\mu}{4^{1-\gamma}}\right)^{-k_0}u(4^{-k_0}x',4^{-k_0}y,4^{-k_0}t).$$
Then
\begin{flalign*}
(i) & \ \inf_{\Gamma_1}\overline{w}\geq -1&&\\ \nonumber
            \text{and} &  &&\\ \nonumber 
(ii) & \ \partial_y \overline{w}\leq nC_0C^{-1}(4^{1+\gamma}\mu)^{-k_0}y^\gamma &&\\ \nonumber 
& \text{(We can choose $k^*$so that $C_0C^{-1}=C_1\mu^{k^*}$, $0<\mu<1$).}
\end{flalign*}

The idea, now, is (see \cite{Athparabolic} and for the elliptic case \cite{Caff79}, \cite{AC}, \cite{MS}  and recently \cite{CFi}) to locate a large region of points $(x',0,t)$ where $|\overline{w}|$ is small (average estimate). If $(\overline{u}-\overline{\psi})$ were actually convex then we would have $\overline{w}\equiv 0$ in at least half a cylinder since $(0,0,0)$ is a free boundary point. Then the appropriate Poisson formula would imply that $|\overline{w}|<1$ in the interior and away from the hyperplane $y=0$. Hence by property $(ii)$ above we would "shoot back" to $y=0$ to have the desired result. Since we are not in this situation we have to correct the argument. 

\textit{Step I (Average estimate):} Set 
$$v(x',y,t):=\overline{u}(x',y,t)-\left[\overline{\psi}(0,0)+\grad_{x'}\overline{\psi}(0,0)\cdot x'+4C_1(4^{1+\gamma}\mu)^{-k_0}\left(|x'|^2-\frac{2n-1}{2(1+\gamma)}y^2-t\right)\right]$$
where $$\overline{\psi}(x',t)=C^{-1}\left(\frac{\mu}{4^{1+\gamma}}\right)^{-k_0}\psi(4^{-k_0}x',4^{-2k_0}t).$$
The function $v$ satisfies equation in (\ref{FBP3}). So by maximum principle applied in $\Gamma_{1/2}$ its nonnegative maximum ($v(0,0,0)=0$) is attained on $\partial_p\Gamma$. Observe that on $\partial_p\Gamma_{1/2}\cap\{y=0\}$ either $\overline{u}=\overline{\psi}$
 or $\overline{u}>\overline{\psi}$ i.e $v<0$ or $\overline{w}\geq 0$, respectively. Hence its maximum must occur on $\partial_p\Gamma_{1/2}$. There appear three cases:

\underline{Case 1:}  The maximum is attained "far from $y=0$" i.e at $\left(x_0',\frac{\sqrt{1+\gamma}}{4n},t_0\right)$.

In this case 
$$(\overline{u}-\overline{\psi})(x_0',\frac{\sqrt{1+\gamma}}{4n},t_0)\geq -\frac{1}{4n}C_1\mu^{k^*}(4^{1+\gamma}\mu)^{-k_0}$$
and by the half-cylinder estimate
$$(\overline{u}-\overline{\psi})(x',\frac{\sqrt{1+\gamma}}{4n},t_0)\geq -\frac{1}{2}C_1\mu^{k^*}(4^{1+\gamma}\mu)^{-k_0}$$
for all $(x',t)\in HQ'_{1/2}(x_0',\frac{\sqrt{1+\gamma}}{4n},t_0)$. "Shooting back" to $y=0$ we observe that $(\overline{u}-\overline{\psi})(x',0,t)>0$ then $\overline{w}(x',0,t)=0$ but when $(\overline{u}-\overline{\psi})(x',0,t)=0$ the above estimate together with property $(ii)$ yields
$$\overline{w}(x',0,t)\geq -\frac{1+16n}{4\sqrt{1+\gamma}}\left(\frac{1+\gamma}{16n^2}\right)^{\gamma/2} C_1\mu^{k^*}(4^{1+\gamma}\mu)^{-k_0}$$
for all $(x',t)\in HQ'_{1/2}(x_0',0,t_0)$.

\underline{Case 2:} The maximum is attained on the lateral sides of $\Gamma_{1/2}$, say at $(x_0',y_0,t_0)$ where $|x_0'|=\frac{1}{2}$.

By the choice of the width of $\Gamma_{1/2}$ we have that 

$$(\overline{u}-\overline{\psi})(x_0',y_0,t_0)\geq \frac{3n-1}{4n}C_1\mu^{k^*}(4^{1+\gamma}\mu)^{-k_0}$$
and by the half-cylinder estimate

$$(\overline{u}-\overline{\psi})(x',y_0,t_0)\geq \frac{2n-1}{4n}C_1\mu^{k^*}(4^{1+\gamma}\mu)^{-k_0}$$
for all $(x',t)\in HQ'_{1/2}(x_0',0,t_0)$. "Shooting back" to $y=0$ we see that by property $(ii)$ we always have that $(\overline{u}-\overline{\psi})(x',0,t)>0$ i.e. 
$$\overline{w}(x',0,t)=0 \ \ \ \text{for all} \ (x',t)\in HQ'_{1/2}(x_0',0,t_0).$$

\underline{Case 3:} The maximum is attained on the bottom of $\Gamma_{1/2}$, say at $(x_0',y_0,-\frac{1}{4})$.  

Again, by the choice of the width of $\Gamma_{1/2}$, we arrive at the same result as is Case 2.

Thus, a fortiori, we have reached at the same conclusion:
$$\textit{In the half cylinder} \ HQ'_{1/2} \ \textit{centered at some point of}\ Q'_{1/2}, \overline{w}(x',0,t)\geq -C_2\mu^{k^*}(4\mu)^{-k_0}.$$

\textit{Step II (Pointwise estimate):} 

Now, choose $\mu\geq\frac{1}{4^{1+\gamma}}$ then
$$0\geq \overline{w}(x',0,t)\geq -C_2\left(\frac{1}{2}\right)^{k^*}\geq -\frac{1}{2}$$
for an appropriate $k^*$. Since $|\overline{w}|\leq 1$ in $\Gamma_1$ and $\overline{w}$ is a solution to the conjugate equation of $\overline{u}$, the Poisson representation formula implies $$|\overline{w}|\leq \theta<1$$ 
in the interior of the domain 
$$D:=\left\{|x'|\leq \frac{1}{2}, \ y=\frac{1}{2}, \ -\frac{1}{4}\leq t\leq 0\right\}.$$
Applying property $(ii)$ once more we show that
$$|\overline{w}|<\theta+C_3\mu^{k^*}(4^{1+\gamma}\mu)^{-k_0}\leq \theta+C_32^{-k^*}=:\mu<1$$
for $(x',y,t)\in \Gamma_{1/2}$ and the proof is complete.

\end{proof}

The next theorem shows that an estimate of the form (\ref{eq3.0}) is enough in order to obtain the following H\"{o}lder estimate. 

\begin{thm}\label{below}
Let $u$ be a solution to (\ref{FBP3}) and $(0,0,0)$ is a free boundary point. Fix $K>0$, $\alpha\in(0,1)$ and assume that 
\begin{equation}\label{below1}
\inf_{\widetilde{Q}_r}y^\gamma u_y\geq -Kr^\alpha
\end{equation}
where $\widetilde{Q}_r:=B_{C_1r}(0)\times[-C_2r,C_2r]\times(-(C_3r)^2,0]$, $r\in(0,1]$ and $C_1$, $C_2$, $C_3$ are fixed positive constants. Then there exists a positive constant $M$ depending on $\gamma$ and $K$ such that 
\begin{equation}\label{below2}
\sup_{\widetilde{Q}_{r/2}}|u-\psi|\leq Mr^{\alpha+2s}.
\end{equation}
\end{thm}

\begin{proof}
Set $v=u-\psi$. First observe that the lower bound in (\ref{below2}) is immediate due to (\ref{below1}) after an integration along $y-$directions. To prove the upper bound, it is enough to show that 
$$v(0,0,-r^2)\leq M r^{\alpha+2s}.$$
Assuming the contrary,  let  $(0,0,t)\in \widetilde{Q}^-_{r/2}$, for $r\in(0,\frac{1}{2})$ since $v$ is bounded, be a point  for which $v(0,0,t)\geq Mr^{\alpha+2s}$ for some large constant $M$, to be chosen. The fact that $$v(x',y,t)\geq v(x',0,t)-K\int_{0}^y\frac{r^\alpha}{\xi^\gamma}d\xi\geq -K\frac{r^{\alpha+2s}}{2s}$$
for every $(x,y,t)\in \widetilde{Q}_r$, implies that $V:=v+K\frac{r^{\alpha+2s}}{2s}$ is non-negative in $\widetilde{Q}_r$. Hence by Proposition \ref{Harnack_dpo} we obtain $$Mr^{\alpha+2s}\leq \sup_{\widetilde{Q}^-_{r/2}}v+K\frac{r^{\alpha+2s}}{2s}\leq C\inf_{\widetilde{Q}^+_{r/4}}v+K\frac{r^{\alpha+2s}}{2s}.$$
That is $$v(0,c_2r,t')\geq c_0Mr^{\alpha+2s}-K\frac{r^{\alpha+2s}}{2s}$$
for some dimensional constant $c_0>0$ and every $t'\in(-\frac{r^2}{4^3},0]$. In addition, for $t_0$ fixed and every $x'\in\{v(x',0,t_0)=0\}$, it holds
$$v(x',y,t_0)-v(x',0,t_0)\leq Cy^2$$
where $C$ depends on $n$, $\gamma$ and the bound of the semiconvexity property. Therefore we obtain, 
$$0=v(0,0,0)\geq v(0,C_2r,0)-Cr^2\geq c_0Mr^{\alpha+2s}-K\frac{r^{\alpha+2s}}{2s}-Cr^2$$
and a contradiction follows for $M$ large enough.

\end{proof}

\section{Optimal regularity of the space derivative}\label{SectionOptimal}

In the present section, we will obtain the optimal space regularity, following the idea presented in \cite{ACM1}, as a consequence of  a parabolic monotonicity formula. For simplicity we take the origin to be a free boundary point. We say that $u$ is $L_\gamma$-caloric if 

\begin{equation}\label{Hg}
L_\gamma u:=\text{div}(|y|^\gamma\grad u)-|y|^\gamma\partial_tu=0
\end{equation}

\begin{preremark} If $u$ is $L_\gamma$-caloric then $y^\gamma u_y$ is $L_{-\gamma}$-caloric.
\end{preremark}

In the present section, as it was done in \cite{ACM1}, subtract the obstacle from our solution in (\ref{FBP3}) which we still denote by $u$. Then our problem becomes

\begin{align}
\begin{cases} \label{FBP4sub}
\div (y^\g \nabla u)-y^\g\partial _t u = f ,\ \ & \text{for} \ (x',y,t) \in \mathbb{R}^{n}_+ \times (0,T] \\
u(x',0,t) \geq 0, \ \ & \text{for} \ (x',t) \in \mathbb{R}^{n-1} \times (0,T] \\
\lim_{y \to 0^+} y^\g \partial _y u(x',y,t)=0, \ \ & \text{for} \ u(x',0,t) > 0 \\
\lim_{y \to 0^+} y^\g \partial _y u(x',y,t) \leq 0,\ \ & \text{for} \ (x',y,t) \in \mathbb{R}^{n}_+ \times (0,T] \\
u(x',0,0)=\phi (x')-\psi(x'),\ \ & \text{for} \ x' \in \R ^{n-1} \end{cases}
\end{align}
where $\phi(x',t)$, $\psi(x',t)$ are smooth functions, satisfying suitable compatibility conditions and with no loss of generality $f(x',y,t):=-L_\gamma\tilde{\psi}(x',t)$, for a proper extension of the smooth obstacle $\psi$ to the whole $\R_+^n$, is independent of the $y$ variable.

The proof of the monotonicity result relies on the following eigenvalue problem. For the proof, see for instance \cite{AC} when $\gamma=0$ or \cite{CFi} for $-1< \gamma<1$. 

\begin{lemma}\label{eigenvalue}
Set $$\lambda_0=\inf_{\substack{w\in H^1(\R^n_+) \\ w=0\ \text{on}\ \R^{n-1}_-}}\frac{\int_{\R^n_{+}}y^{-\gamma}|\grad w(z,-1)|^2e^{-\frac{|z|^2}{4}}dz}{\int_{\R^n_{+}}y^{-\gamma}w^2(z,-1)e^{-\frac{|z|^2}{4}}dz},$$
where $$\R^n_+:=\{z=(x,y)\in \R^n : y>0\}$$
and 
$$\R^{n-1}_-:=\{(x,0):x\in \R^{n-1},\ x_{n-1}<0\}.$$
Then $\lambda_0(1-s)s$.
\end{lemma}
%\subsection{Global Case}\label{global}
We assume that $w$ is a function in $\overline{\R^n_+}\times [-1,0]$ that is $H_{-\gamma}$-caloric in $\R^n_+\times [-1,0]$, where $\R^n_+=\{x=(x',x_n)\in \R^n : x_n>0\}$. We also assume that $w$ has moderate growth at infinity, 
$$\int_{B_R}y^{-\gamma}w^2(z,-1)dz\leq Ce^{\frac{|R|^2}{4+\varepsilon}}$$
for some positive constant $C$, $R$ large and some $\varepsilon>0$. For the present section we assume that $w$ is obtained by solving
$$L_{- \gamma } w = 0, \ \ \ \  \text{in} \ \mathbb{R}^n_+ \times [-1,0] $$
and $ w(x',0,t) = \lim _{y \to 0^+} y^{\gamma } u_y(x',y,t)$. 

From the previous section, we know that $u$ is H\"{o}lder continuous, in particular (\ref{below2}) holds. Our main purpose now is to obtain the optimal space derivative regularity, in the sense that there exists a positive constant $C$ (as in Theorem \ref{thm3.1}) such that around a free boundary point, 
\begin{equation}\label{y-optimal}
\sup_{\widetilde{Q}_{r/2}}|w|\leq Cr^{1-s}.
\end{equation}

We set
\begin{equation}\label{fundamental solution}
G_{\gamma}(x,y,t)= \begin{cases}\frac{c_{n,\gamma }}{t^{\frac{n+ \gamma }{2}}} e^{-\frac{|x|^2+y^2}{4t}} , &  t> 0\cr
 0 &  t\leq 0.
\end{cases}
\end{equation}
 where $\ c_{n, \gamma} = \frac{1}{(4 \pi )^{\frac{n-1}{2}}|\Gamma(\frac{\gamma+1}{2})|}$, where $\Gamma$ here denotes the Gamma function. First we prove a lemma, which uses the normal semi-concavity, the tangential semi-convexity, and the time quasi-convexity. 
\begin{lemma}\label{Lemma5}
Then there exists a $\delta>0$ (depending only on $\alpha$, $s$) and $R_0>0$ small (depending only on $\alpha$, $s$, the semiconvexity constant and the H\"{o}lder constant of $u$)  such that $$(0,0,t)\notin \Gamma(\{w<-r^{\alpha+\delta}\}\cap Q_r')$$
for every $t\in[-r^2,0]$ and $0<r<R_0$. Here $\Gamma(A)$ denotes the convex hull of the set $A$.
\end{lemma} 
\begin{proof}
It is enough to prove it for $t=-r^2$. If $$(x',0,-r^2)\in \{w<-r^{\alpha+\delta}\}$$ then $$u(x',h,-r^2)\leq -\frac{1}{2s}r^{\alpha+\delta}h^{2s}+\frac{nM}{1+\gamma}h^2$$ where $M$ denotes the semi-concavity constant of $u$. In addition, due to the semi-covexity property of $u$ in the tangential directions,
$$u(0,h,-r^2)\leq \sup_{x'} u(x',h,-r^2)+Mr^2.$$ 
Finally, since the origin has been selected as a free boundary point, the $C^{\alpha+2s}$ character of $u$ gives 
$$u(0,h,-r^2)=u(0,h,-r^2)-u(0,0,0)\geq -Ch^{\alpha+2s}.$$
Gathering all the above estimates, we have
$$-Ch^{\alpha+2s}\leq -\frac{1}{2s}r^{\alpha+\delta}h^{2s}+\frac{nM}{1+\gamma}h^2+Mr^2.$$
Now if $h=r^{\alpha+m\delta}$ for $m>1$, a contradiction is obtained by choosing $\delta$ so that $\frac{\alpha(1-\alpha)}{\alpha m-1}<\delta<\frac{2-\alpha(\alpha+2s)}{m(\alpha+2s)}$ provided that $m>\frac{2-(\alpha+2s)\alpha}{\alpha[2-\alpha(\alpha+2s)]}$, and $R_0$ is small enough.
\end{proof}

Now, we provide our monotonicity formula for solutions to the local situation. This type of formulas was first introduced in \cite{AC} for the stationary case with $\g=0$ (for  $\g\in (-1,1)$ see \cite{CFi}) and was extended in \cite{ACM1} to the corresponding evolution problem with $\g=0$.

For simplicity, as it was mentioned at the beginning of the section, we reduce the problem to the case of the zero obstacle.

\begin{lemma}\label{Lemma6non-a}
Let $\delta>0$ and $w$ and $R_0<1$ as above. Set $$\varphi(r)=\frac{1}{r^{2(1-s)}}\int_{-r^2}^0\int_{\R^n_+}y^{-\gamma}|\grad(\eta w)(x',y,\tau)|^2G_{-\gamma}(x',y,-\tau)dxd\tau$$
for $r<1$ where $\eta\in C_0^\infty(B_{2r})$ with $\eta\equiv 1$ and $\eta_{x_n}|_{B_r\cap\R^{n-1}}=0$. There exists a universal constant $C>0$ such that
\begin{enumerate}
\item[(i)] if $2\alpha+\delta>1+\gamma$ then $\varphi(r)\leq C$,
\item[(ii)] if $2\alpha+\delta<1+\gamma$ then $\varphi(r)\leq Cr^{2\alpha+\delta-1-\gamma}$.
\end{enumerate}
\end{lemma}
\begin{proof}

We observe that

\begin{align*}
\div (y^{- \gamma} \nabla ((\eta w)^2)) &= 2 \eta w^2 \div (y^{- \gamma } \nabla \eta ) + 2 \eta ^2 w \div (y^{- \gamma} \nabla w )+ 2 y^{- \gamma} |\nabla \eta |^2 w^2 \\
&\ \ \ + 2  y^{- \gamma} \eta w \nabla \eta \cdot \nabla w +2 y^{- \gamma} \eta ^2 |\nabla w |^2 .
\end{align*}
In addition, we have 
$$|\nabla (\eta w)|^2 = \eta ^2 |\nabla w |^2 + w^2|\nabla \eta |^2 + 2   \eta w \nabla \eta \cdot \nabla w $$
thus
\begin{align*}
y^{- \gamma} |\nabla ( \eta w) |^2 &= \frac{1}{2} \left( \div (y^{- \gamma} \nabla ((\eta w)^2)) - y^{- \gamma } \partial _t (\eta w)^2 \right) - 2 y^{- \gamma} \eta w \nabla \eta \cdot \nabla w  - \eta w^2 \div (y^{- \gamma } \nabla \eta ) . 
\end{align*}
Now we compute
\begin{align*}
\varphi ' (r) &= 2 (s-1) r^{2s-3} \int_{-r^2}^0 \int_{\mathbb{R}^n_+} y^{- \gamma} |\nabla ( \eta w) |^2 G_{- \gamma } \ dxd \tau + 2  r^{2s-1}  \int_{\mathbb{R}^n_+} (y^{- \gamma} |\nabla ( \eta w) |^2 G_{- \gamma } )|_{\tau = -r^2} \ dx \\
&=  (s-1) r^{2s-3} \int_{-r^2}^0 \int_{\mathbb{R}^n_+} \left( \div (y^{- \gamma} \nabla (\eta w)^2) - y^{- \gamma } \partial _t (\eta w)^2 \right) G_{- \gamma } \ dx d \tau \\
&\ \ \ - 2 (s-1) r^{2s-3} \int_{-r^2}^0 \int_{\mathbb{R}^n_+} \left( 2 y^{- \gamma} \eta w \nabla \eta \cdot \nabla w  + \eta w^2 \div (y^{- \gamma } \nabla \eta ) \right) G_{- \gamma } \ dx d \tau \\
&\ \ \ + 2  r^{2s-1}  \int_{\mathbb{R}^n_+} (y^{- \gamma} |\nabla ( \eta w) |^2 G_{- \gamma } )|_{\tau = -r^2} \ dx.
\end{align*}
We now integrate by parts, using a standard mollification argument, to obtain
\begin{align*}
\varphi _{\epsilon} ' (r) &= (1-s) r^{2s-3} \int_{-r^2}^0 \int_{\mathbb{R}^n_+} \left( y^{- \gamma} \nabla ((\eta w)^2) \nabla G_{- \gamma} + y^{- \gamma } \partial _t (\eta w)^2 G_{- \gamma } \right)  \ dx d \tau \\
&\ \ \ - (1-s) r^{2s-3} \int_{-r^2}^0 \int_{\mathbb{R}^{n-1}} \left( y^{- \gamma} ((\eta w)^2)_{\nu } G_{- \gamma } \right) |_{y = \epsilon }  \ dx' d \tau \\
&\ \ \ - 2 (s-1) r^{2s-3} \int_{-r^2}^0 \int_{\mathbb{R}^n_+} \left( 2 y^{- \gamma} \eta w \nabla \eta \cdot \nabla w  + \eta w^2 \div (y^{- \gamma } \nabla \eta ) \right) G_{- \gamma } \ dx d \tau \\
&\ \ \ + 2  r^{2s-1}  \int_{\mathbb{R}^n_+} (y^{- \gamma} |\nabla ( \eta w) |^2 G_{- \gamma } )|_{\tau = -r^2} \ dx.
\end{align*}
We again integrate by part to obtain
\begin{align} \label{Lemma6non-a_1}
\varphi _{\epsilon} ' (r) &= -(1-s) r^{2s-3} \int_{-r^2}^0 \int_{\mathbb{R}^n_+} \left( (\eta w)^2 \div (y^{- \gamma} \nabla G_{- \gamma}) - y^{- \gamma } \partial _t (\eta w)^2 G_{- \gamma } \right)  \ dxd \tau \\ \nonumber
&\ \ \ + (1-s) r^{2s-3} \int_{-r^2}^0 \int_{\mathbb{R}^{n-1}} \left( y^{- \gamma} (\eta w)^2 (G_{- \gamma })_{\nu } \right) |_{y = \epsilon }  \ dx' d \tau \\ \nonumber
&\ \ \ - (1-s) r^{2s-3} \int_{-r^2}^0 \int_{\mathbb{R}^{n-1}} \left( y^{- \gamma} ((\eta w)^2)_{\nu } G_{- \gamma } \right) |_{y = \epsilon }  \ dx' d \tau \\ \nonumber
&\ \ \ - 2 (s-1) r^{2s-3} \int_{-r^2}^0 \int_{\mathbb{R}^n_+} \left( 2 y^{- \gamma} \eta w \nabla \eta \cdot \nabla w  + \eta w^2 \div (y^{- \gamma } \nabla \eta ) \right) G_{- \gamma } \ dx d \tau \\ \nonumber
&\ \ \ + 2  r^{2s-1}  \int_{\mathbb{R}^n_+} (y^{- \gamma} |\nabla ( \eta w) |^2 G_{- \gamma } )|_{\tau = -r^2} \ dx. 
\end{align}
Now, we will estimate the second and the third terms in the above expression. To do so, let us denote by
$$I_{\epsilon } := \int_{-r^2}^0 \int_{\mathbb{R}^{n-1}} \left( y^{- \gamma} ((\eta w)^2)_{y } G_{- \gamma } \right) |_{y = \epsilon }  \ dx' d \tau $$
and
$$J_{\epsilon } := \int_{-r^2}^0 \int_{\mathbb{R}^{n-1}} \left( y^{- \gamma} (\eta w)^2 (G_{- \gamma })_{y} \right) |_{y = \epsilon }  \ dx' d \tau .$$
We will follow the proof of Lemma 4.6 in \cite{CFi}, in particular we will show that $I_{\epsilon }$ is bounded below by a certain power of $r$, while $J_{\epsilon }$ goes to zero whenever $\epsilon \to 0^+$. Firstly, we observe that $w$ is obtained by solving
\begin{align*}
\begin{cases}
L_{- \gamma } w = 0, &\ \ \ \  \text{in} \ \mathbb{R}^n_+ \times \mathbb{R}^+ \\
w(x',0,t) = \lim _{y \to 0^+} y^{\gamma } u_y(x',y,t),  &\ \ \ \  \text{in} \ \mathbb{R}^{n-1} \times \mathbb{R}^+. 
\end{cases}
\end{align*}
Therefore, using Lemma  \ref{lemma3.1}, we have that
$$w(x',y,\tau) \geq w(x',0,\tau) - \frac{nC_0}{1 + \gamma} y^{1+ \gamma }$$
for all $x' \in \mathbb{R}^{n-1},\ y>0$ and $\tau$ fixed, or equivalently
$$w^2(x',y,\tau) - w^2(x',0,\tau) \geq - C y^{1+ \gamma } (r+y)^{\alpha }$$
for $x' \in B_r \cap \mathbb{R}^{n-1}$ and $y>0$, since $w$ is H\"{o}lder continuous and is nonnegative. Now, consider a change of variables $\xi = y^{1+\gamma
 }$ and let $\tilde{w}(x',\xi ,\tau):= w(x',y,\tau)$. Then 
\begin{equation} \label{Lemma6non-a_2}
\tilde{w}^2(x',\xi,\tau) - \tilde{w}^2(x',0,\tau) \geq - C_2 \xi  (r+\xi ^{\frac{1}{1+\gamma }})^{\alpha }
\end{equation}
for any $\tau$ fixed, $x' \in B_r$ and $\xi >0$. In addition, $y^{-\gamma }(w^2)_y=C_3 (\tilde{w}^2)_{\xi}$ for a constant $C_3$ depending on $\gamma $. Therefore, we are able to estimate the average of the spatial integral of $I_{\epsilon }$ as follows
\begin{align} \label{Lemma6non-a_3}
A &= \frac{C_{n,\gamma }}{\epsilon } \int_0^{\epsilon } \int_{\mathbb{R}^{n-1}} (\eta \tilde{w})_{\xi}^2(x',\xi ,\tau ) \frac{e^{\frac{|x'|^2+C_4 \xi ^{\frac{2}{1+\gamma }}}{4 \tau}}}{(-\tau )^{\frac{n-\gamma }{2}}} \ dx' d \xi \\ \nonumber
&= \frac{C_{n,\gamma }}{\epsilon } \int_{\mathbb{R}^{n-1}} \frac{1}{(-\tau )^{\frac{n-\gamma }{2}}} \left( \eta \tilde{w} ^2(x',\epsilon ,\tau ) e ^{\frac{|x'|^2+C_4 \epsilon ^{\frac{2}{1+\gamma }}}{4 \tau}} - \eta \tilde{w}^2(x',0,\tau ) e ^{\frac{|x'|^2}{4\tau }} \right) \ dx'  \\ \nonumber
&\ \ \ - \frac{C_{n,\gamma }}{\epsilon } \int_0^{\epsilon } \int_{\mathbb{R}^{n-1}} \frac{1}{(-\tau )^{\frac{n-\gamma }{2}}} \eta \tilde{w}^2(x',\xi ,\tau ) \frac{d}{d\xi } \left( e ^{\frac{|x'|^2+C_4 \xi ^{\frac{2}{1+\gamma }}}{4 \tau}} \right) \ dx' d \xi 
\end{align}
for any $\tau \in (-r^2,0)$ fixed. Observe that the last term in (\ref{Lemma6non-a_3}) is positive which gives
\begin{align} \label{Lemma6non-a_4}
&\frac{C_{n,\gamma }}{\epsilon } \int_0^{\epsilon } \int_{\mathbb{R}^{n-1}} (\eta \tilde{w})_{\xi}^2(x',\xi ,\tau ) \frac{e ^{\frac{|x'|^2+C_4 \xi ^{\frac{2}{1+\gamma }}}{4 \tau}}}{(-\tau )^{\frac{n-\gamma }{2}}} \ dx' d \xi \\ \nonumber 
&\geq \frac{C_{n,\gamma }}{\epsilon } \int_{\mathbb{R}^{n-1}} \frac{1}{(-\tau )^{\frac{n-\gamma }{2}}} \left( \eta \tilde{w} ^2(x',\epsilon ,\tau ) e ^{\frac{|x'|^2+C_4 \epsilon ^{\frac{2}{1+\gamma }}}{4 \tau}} - \eta \tilde{w}^2(x',0,\tau ) e ^{\frac{|x'|^2}{4\tau }} \right) \ dx' 
\end{align}
for any $\tau \in (-r^2,0)$ fixed. Now
$$\tilde{w}^2(x',\epsilon ,\tau ) - \tilde{w}^2(x',0,\tau ) \geq - C_2 \epsilon  (r+\epsilon ^{\frac{1}{1+\gamma }})^{\alpha }$$
by (\ref{Lemma6non-a_2}), therefore 
\begin{align} \label{Lemma6non-a_5}
&\frac{C_{n,\gamma }}{\epsilon } \int_0^{\epsilon } \int_{\mathbb{R}^{n-1}} (\eta \tilde{w})_{\xi}^2(x',\xi ,\tau ) \frac{e ^{\frac{|x'|^2+C_4 \xi ^{\frac{2}{1+\gamma }}}{4 \tau}}}{(-\tau )^{\frac{n-\gamma }{2}}} \ dx' d \xi \\ \nonumber 
&\geq \frac{C_{n,\gamma }}{\epsilon } \int_{\mathbb{R}^{n-1}} \frac{1}{(-\tau )^{\frac{n-\gamma }{2}}} \left(  e ^{\frac{|x'|^2+C_4 \epsilon ^{\frac{2}{1+\gamma }}}{4 \tau}} \eta ^2 \tilde{w} ^2(x',0,\tau ) -  e ^{\frac{|x'|^2}{4\tau }} \eta ^2 \tilde{w}^2(x',0,\tau )  \right) \ dx' \\ \nonumber
&\ \ \ - \frac{C_{n,\gamma }}{\epsilon } \int_{\mathbb{R}^{n-1}} \frac{1}{(-\tau )^{\frac{n-\gamma }{2}}} C_2 e ^{\frac{|x'|^2+C_4 \epsilon ^{\frac{2}{1+\gamma }}}{4 \tau}} \eta ^2 \epsilon (r+\epsilon ^{\frac{1}{1+\gamma }})^{\alpha} \ dx' \\ \nonumber
&\geq \frac{C_{n,\gamma }}{\epsilon } \int_{\mathbb{R}^{n-1}} \frac{\eta ^2 |x'|^{2\alpha }}{(-\tau )^{\frac{n-\gamma }{2}}} \left( e ^{\frac{|x'|^2+C_4 \epsilon ^{\frac{2}{1+\gamma }}}{4 \tau}} - e ^{\frac{|x'|^2}{4\tau }} \right) \ dx' - C_{n,\gamma } \int_{\mathbb{R}^{n-1}} e ^{\frac{|x'|^2+C_4 \epsilon ^{\frac{2}{1+\gamma }}}{4 \tau}} \eta ^2  (r+\epsilon ^{\frac{1}{1+\gamma }})^{\alpha} \ dx' \\ \nonumber
&\ \ \ + \frac{C_{n,\gamma }}{\epsilon } \int_{\mathbb{R}^{n-1}} \frac{1}{(-\tau )^{\frac{n-\gamma }{2}}} \eta ^2 \tilde{w}^2(0,0,\tau ) \left( e ^{\frac{|x'|^2+C_4 \epsilon ^{\frac{2}{1+\gamma }}}{4 \tau}} - e ^{\frac{|x'|^2}{4\tau }} \right) \ dx'
\end{align}
for every $\tau \in (-r^2,0)$ fixed, since $\tilde{w}$ is H\"older continuous with H\"older exponent $\alpha $. Now, if we pass to the limit, we obtain
$$\int_{-r^2}^0 \int_{\mathbb{R}^{n-1}} \frac{\eta ^2 (x')}{(-\tau )^{\frac{n-\gamma }{2}}} e ^{\frac{|x'|^2+C_4 \epsilon ^{\frac{2}{1+\gamma }}}{4 \tau}} (r+\epsilon ^{\frac{1}{1+\gamma }})^{\alpha} \ dx' d\tau \  \to Cr^{1+\gamma +\alpha }$$
as $\epsilon \to 0^+$, where $C$ is a universal constant depending on $n$ and $\gamma $. This fact and the expression (\ref{Lemma6non-a_5}) above shows that
$$\limsup _{y \to 0^+} \int_{-r^2}^0 \int_{\mathbb{R}^{n-1}} y^{-\gamma } ((\eta w)^2)_y G_{-\gamma } \ dx' d\tau \geq -C r^{1+\gamma +\alpha }$$
for some constant $C=C(n,\gamma )$ and every $r \geq 0$.

\par To estimate $J_{\epsilon }$, we simply compute $\partial _y G_{-\gamma }$ and make use of the H\"older regularity of $w$ to gain control of the integral. In particular,
\begin{align*}
\int_{\mathbb{R}^{n-1}} (\eta w)^2 y^{-\gamma } (G_{-\gamma })_y \ dx' &= -C_{n,\gamma } \int_{\mathbb{R}^{n-1}} \frac{1}{2} \frac{1}{(-\tau )^{1+\frac{n-\gamma }{2}}} (\eta w)^2 y^{-\gamma +1} e^{\frac{|x'|^2+|y|^2}{4\tau }} \dx' \\
&\leq C \int_{\mathbb{R}^{n-1}} \frac{1}{|\tau |^{1+\frac{n-\gamma }{2}}} |\eta ^2 w^2(x',y,\tau ) - \eta ^2 w(0,0,\tau )| y^{-\gamma +1} e^{\frac{|x'|^2+|y|^2}{4\tau }}  \ dx' \\
&\ \ \ + C \int_{\mathbb{R}^{n-1}} \frac{1}{|\tau |^{1+\frac{n-\gamma }{2}}}  \eta ^2 w(0,0,\tau ) y^{-\gamma +1} e^{\frac{|x'|^2+|y|^2}{4\tau }}  \ dx'
\end{align*}
for every $\tau \in (-r^2,0)$ fixed, $y>0$ where $C$ is a positive constant depending on $n$ and $\gamma $. Therefore
\begin{align*}
&\left| \int_{-r^2}^0 \int_{\mathbb{R}^{n-1}} (\eta w)^2 y^{-\gamma } (G_{-\gamma })_y \ dx' d\tau \right| \\
&\leq C  \int_{-r^2}^0 \int_{\mathbb{R}^{n-1}} \frac{1}{|\tau |^{1+\frac{n-\gamma }{2}}} |\eta ^2 w^2(x',y,\tau ) - \eta ^2 w(0,0,\tau )| y^{-\gamma +1} e^{\frac{|x'|^2+|y|^2}{4\tau }}  \ dx' d\tau \\
&\ \ \ + C \int_{-r^2}^0 \int_{\mathbb{R}^{n-1}} \frac{1}{|\tau |^{1+\frac{n-\gamma }{2}}}  \eta ^2 w(0,0,\tau ) y^{-\gamma +1} e^{\frac{|x'|^2+|y|^2}{4\tau }}  \ dx' 
\end{align*}
for every $y>0$. In the above expression, the second term goes to zero as $y \to 0^+$ while the first term is controlled by $y^{2\alpha }$ (due to Lemma  \ref{lemma3.1}), for $y$ small and as such, it also goes to zero as $y \to 0^+$.
\par With the above estimates at hand, we let $\epsilon \to 0^+$ in (\ref{Lemma6non-a_1}) to obtain  
\begin{align*}
\varphi '(r) &\geq -(1-s) r^{2s-3} \int_{-r^2}^0 \int_{\mathbb{R}^n_+} ( (\eta w)^2 \div (y^{-\gamma } \nabla G_{-\gamma } ) - y^{-\gamma } \partial _t (\eta w)^2 G_{-\gamma } ) \ dxd\tau \\
&\ \ \ - 2(s-1) r^{2s-3} \int_{-r^2}^0 \int_{\mathbb{R}^n_+} ( 2 y^{-\gamma } \eta w \nabla \eta \nabla w + \eta w^2 \div (y^{-\gamma } \nabla \eta ) ) G_{-\gamma } \ dx d\tau \\
&\ \ \ + 2 r^{2s-1} \int_{\mathbb{R}_+^n} ( y^{-\gamma } |\nabla (\eta w)^2| G_{-\gamma } ) | _{\tau = -r^2} \ dx - C r^{1+\gamma +\alpha }
\end{align*}
for every $r \in (0,1)$, where $C$ is a universal constant. We integrate by parts on the first term to obtain
\begin{align*}
\varphi '(r) &\geq -(1-s) r^{2s-3} \int_{-r^2}^0 \int_{\mathbb{R}^n_+}  (\eta w)^2 ( \div (y^{-\gamma } \nabla G_{-\gamma } ) + y^{-\gamma } \partial _t (G_{-\gamma }) ) \ dx d\tau \\
&\ \ \ - (1-s) r^{2s-3}  \int_{\mathbb{R}^n_+}  y^{-\gamma } (\eta w)^2 (x',y,-r^2) G_{-\gamma }(x,y,r^2) \ dx  \\
&\ \ \ - 2(s-1) r^{2s-3} \int_{-r^2}^0 \int_{\mathbb{R}^n_+} ( 2 y^{-\gamma } \eta w \nabla \eta \nabla w + \eta w^2 \div (y^{-\gamma } \nabla \eta ) ) G_{-\gamma } \ dx d\tau \\
&\ \ \ + 2 r^{2s-1} \int_{\mathbb{R}_+^n} ( y^{-\gamma } |\nabla (\eta w)^2| G_{-\gamma } ) | _{\tau = -r^2} \ dx - C r^{1+\gamma +\alpha }.
\end{align*}
Since $w(0,0,0)=0$ we end up,
\begin{align*}
\varphi '(r) &\geq -(1-s) r^{2s-3} \int_{\mathbb{R}^n_+} y^{-\gamma } (\eta w)^2 (x',y,-r^2)  G_{-\gamma } (x,y,r^2)  \ dx  \\
&\ \ \ +2 r^{2s-1}  \int_{\mathbb{R}^n_+} y^{-\gamma } |\nabla (\eta w)|^2 (x',y,-r^2) G_{-\gamma } (x,y,r^2) \ dx \\
&\ \ \ - (s-1) r^{2s-3} \int_{-r^2}^0 \int_{\mathbb{R}^n_+}  y^{-\gamma } \nabla \eta ^2 \nabla w^2 G_{-\gamma }  \ dx d\tau \\
&\ \ \ + 2(1-s) r^{2s-3} \int_{-r^2}^0 \int_{\mathbb{R}^n_+}  \eta w^2 \div (y^{-\gamma } \nabla \eta )  G_{-\gamma } \ dx  d\tau - C r^{1+\gamma +\alpha }
\end{align*}
and finally,
\begin{align*}
\varphi '(r) &\geq -(1-s) r^{2s-3} \int_{\mathbb{R}^n_+} y^{-\gamma } (\eta w)^2 (x',y,-r^2)  G_{-\gamma } (x',y,r^2)  \ dx \\
&\ \ \ +2 r^{2s-1}  \int_{\mathbb{R}^n_+} y^{-\gamma } |\nabla (\eta w)|^2 (x',y,-r^2) G_{-\gamma } (x',y,r^2) \ dx   - C r^{1+\gamma +\alpha }.
\end{align*}
Now consider the truncated function $\overline{w} := -(w+r^{\alpha +\delta })^-$ and note that
\begin{align*}
\int_{\mathbb{R}^n_+} y^{-\gamma } |\nabla (\eta \overline{w} ) (x',y,-r^2) |^2 G_{-\gamma } (x',y,r^2)  \ dx  \leq  \int_{\mathbb{R}^n_+} y^{-\gamma } |\nabla (\eta w) (x',y,-r^2) |^2 G_{-\gamma } (x',y,r^2)  \ dx.
\end{align*}
Therefore
\begin{align*}
\varphi '(r) &\geq -(1-s) r^{2s-3} \int_{\mathbb{R}^n_+} y^{-\gamma } (\eta (w-\overline{w})+\eta \overline{w})^2 (x',y,-r^2) G_{-\gamma }(x',y,r^2) \ dx \\
&\ \ \ + 2 r^{2s-1} \int_{\mathbb{R}^n_+} y^{-\gamma } |\nabla (\eta w)|^2 (x,y,-r^2)  G_{-\gamma } (x',y,r^2)  \ dx -C r^{1+\gamma +\alpha }
\end{align*}
hence
\begin{align*}
\varphi '(r) &\geq -(1-s) r^{2s-3} \int_{\mathbb{R}^n_+} y^{-\gamma } \eta ^2 \left[ (w-\overline{w})^2 +2\overline{w} (w-\overline{w}) \right] G_{-\gamma }(x',y,r^2) \ dx -C r^{1+\gamma +\alpha } \\
&\geq -C r^{2\alpha +\delta -\gamma -2} -C r^{1+\gamma +\alpha } 
\end{align*}
since $2s=1-\gamma $. Therefore
$$\varphi ' (r) \geq -C -C r^{2\alpha +\delta -\gamma -1}$$
and
$$\varphi (1) -\varphi (r) \geq -C + C  r^{2\alpha +\delta -\gamma -1}.$$
Since $\varphi (1)$ is universally bounded, the proof is complete.
\end{proof}

Next, we state our main result of this section:

\begin{thm}\label{Theorem5}
Let $u$ the solution of (\ref{FBP3}), then $u$ is $C^{1,s}$ up to the hyperplane $\R^{n-1}$, in the sense that (\ref{y-optimal}) holds.
\end{thm}
\begin{proof}
Let $w$ and $\overline{w}$ be as in the proof of Lemma \ref{Lemma6non-a}. Observe that $\overline{w}$, by its definition, is $L_{-\g}-$subsolution in $\{y>0\}$ and can be extended to be $L_{-\g}-$subsolution globally (see for instance \cite{CFi}). Now fix $\tau>0$, choose $R>0$ large enough and $\varepsilon<\tau$. We define a cut-off function $\eta=\eta(X)$ so that $\text{supp}\eta\in B_{R+1}(0)$, $\eta \equiv 1$ on $B_R(0)$ and $|\grad\eta|\leq C$.  Observe that
$$\div(y^{-\gamma}\grad(\eta\overline{w})^2)-y^{-\gamma}\partial_\xi(\eta\overline{w})^2=2y^{-\gamma}\eta^2|\grad \overline{w}|^2+8y^{-\gamma}\overline{w}\eta\grad\overline{w}\grad\eta
+2(\eta\Delta\eta+y^{-\gamma}|\grad\eta|^2)\overline{w}^2$$
\begin{equation}\label{thm5-1}
+2\eta^2\overline{w}(\div(y^{-\gamma}\grad\overline{w})-y^{-\gamma}\partial_\xi\overline{w}).
\end{equation}

An integration by parts along with the fact that $\eta$ is compactly supported, gives
$$2\int_{-\tau}^{-\varepsilon}\int_{\R^n}|y|^{-\gamma}\eta^2|\grad\overline{w}|^2G_{-\gamma}(X,-\xi)dz d\xi=-\int_{\R^n}|y|^{-\gamma}\eta^2\overline{w}^2G_{-\gamma}(X,\varepsilon)dz +\int_{\R^n}|y|^{-\gamma}\eta^2\overline{w}^2G_{-\gamma}(X,\tau)dz $$ 
$$-8\int_{-\tau}^{-\varepsilon}\int_{\R^n}|y|^{-\gamma}\overline{w}
\eta\grad\eta\grad
\overline{w}G_{-\gamma}(x,-\xi)dz d\xi-2\int_{-\tau}^{-\varepsilon}\int_{\R^n}(\eta\div(|y|^{-\gamma}\grad{\eta})+|y|^{-\gamma}|\grad\eta|^2)\overline{w}^2G_{-\gamma}(x,-\xi)dz d\xi$$
\begin{equation}\label{thm5-2}
-2\int_{-\tau}^{-\varepsilon}\int_{\R^n}\eta^2\overline{w}(\div(|y|^{-\gamma}\overline{w})-|y|^{-\gamma}\partial_{\xi}\overline{w})G_{-\gamma}(x,-\xi)dz d\xi.
\end{equation}
Observe that $$\int_{-\tau}^{-\varepsilon}\int_{\R^n}|y|^{-\gamma}\overline{w}\eta|\grad\eta||\grad\overline{w}|G_{-\g}(x,-\xi)dz d\xi\leq C\int_{-\tau}^{-\varepsilon}\int_{B_{R+1}\setminus B_{R}}|y|^{-\gamma}|\overline{w}||\grad\overline{w}|\frac{e^{-R^2/4|\xi|}}{|\xi|^{n/2-\gamma/2}}dz d\xi$$
$$\leq Ce^{-R^2/4+\varepsilon_0}\int_{-\tau}^{0}\int_{B_{R+1}\setminus B_{R}}|y|^{-\gamma}|\overline{w}||\grad\overline{w}|dz d\xi.$$
Using Cauchy-Schwartz, we conclude that the last three terms on the right hand side of (\ref{thm5-2}) behave the same, in particular they decay to zero as $R\rightarrow \infty$. Therefore we conclude that
$$(\eta\overline{w})^2(0,0)\leq \int_{\R^n}|y|^{-\gamma}(\eta\overline{w})^2G_{-\gamma}(X,\tau)dX$$
or, after rescaling,
\begin{equation}\label{thm5-3}
(\eta\overline{w})^2(X,t)\leq \int_{\R^n}|z_{n}|^{-\gamma}(\eta\overline{w})^2(z,\tau)G_{-\g}(X-z,t-\tau)dz
\end{equation}
for every $(X,t)\in Q^+_{r/2}$ and $-r^2<\tau<-\frac{r^2}{2}$. By weighted Poincar\'{e} inequality for Gaussian measures we have that
\begin{equation}\label{thm5-4}
\int_{\R^n}|z_{n}|^{-\gamma}(\eta\overline{w})^2(z,\tau)G(X-z,t-\tau)dz\leq 2 |\tau|\int_{\R^n}|z_{n}|^{-\gamma}|\grad(\eta\overline{w})(z,\tau)|^2G_{-\g}(X-z,t-\tau)dz
\end{equation}
for $(X,t)\in Q^+_{r/2}$ and $-r^2<s<-\frac{r^2}{2}$. Combine (\ref{thm5-3}) and (\ref{thm5-4})
to obtain
\begin{equation}\label{thm5-5}
(\eta\overline{w})^2(X,t)\leq C |\tau|\int_{\R^n}|z_{n}|^{-\gamma}|\grad(\eta\overline{w})(z,s)|^2G_{-\g}(X-z,t-\tau)dz
\end{equation}
for every $(X,t)\in Q^+_{r/2}$ and $-r^2<\tau<-\frac{r^2}{2}$. An integration with respect to $\tau$ in (\ref{thm5-5}) shows that
$$(\eta\overline{w})^2(x,t)\leq C\int_{-r^2}^{-r^2/2}\int_{\R^n}|z_{n}|^{-\gamma}|\grad(\eta\overline{w})(z,s)|^2G_{-\g}(X-z,t-\tau)dzd\tau$$
for every $(X,t)\in Q^+_{r/2}$. Now the dichotomy for $\varphi(r)$ in Lemma \ref{Lemma6non-a} provides an $r^{\frac{1+\g}{2}}$ decay for $w$ at free boundary points, as in the proof of Theorem 5 in \cite{AC} (see also Proposition 4.9 in \cite{CFi} for $\g\neq 0$).
\end{proof}

%\begin{cor}
%The solution $u$ of the problem (\ref{FBP1}) is in the class $C^{1,s}$.
%\end{cor}
%\begin{proof}
%From Theorem \ref{Theorem5}, we know that $u$ (or $u-\psi$ in the case of the nonzero obstacle) has a $C^{1,s}$ decay at free boundary points. The equivalence of the problems (\ref{FBP1}) and (\ref{FBP3}) concludes the proof.
%\end{proof}

%\begin{preremark}\label{alternative}
%An alternative proof of the optimal space regularity of the solution can be done using Parabolic Almgren's frequency formula, following the idea of \cite{CSS} (see also \cite{DGPT2013} for the case of caloric functions). A frequency formula in the parabolic setting (\ref{Hg}), as it associates all considerations and calculations in the corresponding weighted Gaussian space (with $|y|^{\g}$- weight), will be the context of a forthcoming note. 
%\end{preremark}

\section{Continuity of the time derivative}\label{SectionTime}

As it was mentioned in the introduction the time derivative can be, in general, discontinuous across free boundary points, but if we assume that there is enough "mass" present from the past near a free boundary point i.e. having a parabolic positive density we can, as it was done in \cite{ACM1}, prove that the time derivative is H\"{o}lder continuous in a neighborhood of such a point. This is the purpose of the present section. We prove it by employing the "hole filling" method introduced for systems by K.-O. Widman (see \cite{Wid}) and thus we extent our previous result (see \cite{ACM1}). Again, we subtract the obstacle from the solution i.e. we work with the zero obstacle. First the density definition

\begin{defn}\label{positive density}
A free boundary point $(x_0',0,t_0)$ is of positive parabolic density with respect to the coincidence set if there exist positive constants $c>0$ and $r_0>0$ such that $$|Q_r'(x'_0,0,t_0)\cap\{u=0\}|\geq c|Q'_r(x'_0,0,t_0)|$$  for all $r<r_0$.
\end{defn}

A key r\^{o}le in the proof of Theorem \ref{Theorem6} below is played by $F_\gamma$, the fundamental solution of $L_\gamma$. Since the operator is invariant under translations in $x'$ and $t$ variables then, if its pole is located on the hyperplane then $F_\gamma=G_\gamma$ is given just by (\ref{fundamental solution})). But if its pole is off the hyperplane, say at $(0,1,0)$, then $F_\gamma=G_\gamma h_\gamma$ 
where $h_\gamma $ is a positive bounded function of $z=\frac{y}{t}$. More precisely, let $g_\gamma(y,t)$ be defined as a quotient of the fundamental solution $F_\gamma$ of $L_\gamma$ with pole $(0,1,0)$ over $G_\gamma(x,y-1,t)$ then away from the pole it satisfies the equation:
$$g_{yy}+(\frac{\gamma}{y}-\frac{y-1}{t})g_y+\frac{\gamma}{2ty}g-g_t=0.$$ 
Therefore looking for one dimensional solutions $h$ in the variable $z=\frac{y}{t}$ we obtain the ordinary differential equation
$$h^{''}_\gamma+(\frac{\gamma}{z}+1)h'_\gamma+\frac{\gamma}{2z}h_\gamma=0$$
which admits for any $-1<\gamma<1$ positive bounded solution for $0<z<\infty$ such that $h(0)=1$. For $\gamma=0$, as expected, $h_\gamma=1$ is a solution. Also, an explicit solution is known in the limiting case $\gamma=1$ i.e.  
$$h(z)=\frac{1}{2\pi}\int_0^{2\pi} e^{-\frac{z}{2}(1-cos\theta)}d\theta.$$  

Now, we are ready to prove the main result of this section :

\begin{thm}\label{Theorem6}
Let $(x'_0,0,t_0)$ be a free boundary point of positive parabolic density with respect to the coincidence set to problem (\ref{FBP3}). Then the time derivative of the solution  is H\"{o}lder continuous in a neighborhood of $(x'_0,0,t_0)$.
\end{thm} 

\begin{proof}As in \S \ref{SectionOptimal}, we extend the obstacle and subtract it from the solution which we still denote by $u$. Since $u_t$ is zero on the coincidence set, $\{u=0\}$, it suffices to prove that $u_t^+$, as well as $u_t^-$, are H\"{o}lder continuous. Actually, we will show that $u_t^+$ and $u_t^-$ decay to zero in parabolic cylinders shrinking to the free boundary point $(x'_0,0,t_0)$. We consider the penalized solution $u^\varepsilon$ of (\ref{FBP4sub}) in $Q_r^+(x_0,t_0)$ with $r<r_0$, where $r_0$ is as in Definition \ref{positive density}. For simplicity we take $(x_0,t_0)=(0,0)$ and $r=1$. Our penalized solution is constructed to satisfy

\begin{equation}\label{penmodel2.1}
\begin{cases}  
\div(y^\gamma \grad u^\varepsilon)-y^\gamma\partial_t u^\varepsilon=f^\varepsilon,   &{\rm{in}} \ \  Q_1^+ \cr  \lim_{y\to 0^+} y^\gamma \partial_y u^\varepsilon
=\beta_\varepsilon(u^\varepsilon)  &{\rm{on}} \ \  Q_1' 
%\cr
%u^\varepsilon(x',0)=\phi^\varepsilon(x')-\psi^\varepsilon(x',0)+\varepsilon %&{\rm{on}} \ \  \R^{n-1}
\end{cases}
\end{equation}
where 
%$\phi^\varepsilon$, $\psi^\varepsilon$ are smooth functions (with compact support in the case of the 
%whole $\R^{n-1}$), 
$\beta_\varepsilon(s)=-e^{\frac{\varepsilon}{s-\varepsilon}}\chi_{s\leq \varepsilon}(s)$ 
with $\psi^\varepsilon\rightarrow \psi$,  $\phi^\varepsilon\rightarrow \phi$ (locally) uniformly as $\varepsilon\rightarrow 0$. 

Differentiate (\ref{penmodel2.1})  with respect to $t$ to have
\begin{equation}\label{penalized equation}
\begin{cases}  
\div(y^\gamma \grad v^\varepsilon)-y^\gamma\partial_t v^\varepsilon= f^\varepsilon_t,   &{\rm{in}} \ \  Q_1^+ \cr  \lim_{y\to 0^+} y^\gamma \partial_y v^\varepsilon
=\beta'_\varepsilon(u^\varepsilon)v^\varepsilon  &{\rm{on}} \ \  Q_1' \cr
\end{cases}
\end{equation}
%$\Delta u^\varepsilon-\partial_t u^\varepsilon=f^\varepsilon$   in  $Q_1$, $\partial_\nu u^\varepsilon=\beta_\varepsilon(u^\varepsilon)$  on $Q_1'$
where  $v^\varepsilon:=(u^\varepsilon)_t$ and as before $y=x_n$. 
For any $(\xi,\tau)\in Q^+_\frac{1}{5}$ we want to multiply the equation by an appropriate test function and integrate by parts over the set 
$Q_\frac{3}{5}^+(\xi,\tau):=Q_\frac{3}{5}(\xi,\tau)\cap \{x_n\geq0\}\subset Q^+_1$. 
This will give us an estimate which will be iterated to yield the desired result.

As it was done in Theorem 4.8 of \cite{ACM1}, 
$\zeta^2 F_{\gamma,\delta}^{(\xi,\tau)} (v^{\varepsilon})^+$
is the appropriate test function where $(v^{\varepsilon})^+=max\{v^{\varepsilon},0\}$, 
$ F_{\gamma,\delta} ^{(\xi,\tau)} (x,t)$ is a smoothing of the fundamental solution $F_\gamma(x,t)$ with pole at $(\xi,\tau)$ and
$\zeta(x,t)$ is a smooth function supported in 
$Q_\frac{3}{5}^+(\xi,\tau)$ such that 
$\zeta\equiv 1$ for every $(x,t)\in Q_\frac{2}{5}^+(\xi,\tau)$, $|\grad\zeta|\leq c$\ with\ 
$supp(\grad\zeta)\subset (B_\frac{3}{5}^+(\xi,\tau)\setminus B_\frac{2}{5}^+(\xi,\tau))\times(\tau-\frac{9}{25},\tau],
\ 0\leq \zeta_t \leq c$  with
$\ supp (\zeta_t)\subset B_\frac{3}{5}(\xi,\tau) \times (\tau-\frac{9}{25},\tau-\frac{4}{25})$.

Therefore we multiply the equation in (\ref{penalized equation}) by $\zeta^2 F_{\gamma,\delta}^{(\xi,\tau)} (v^{\varepsilon})^+$ 
%(or by $\zeta^2 F_{\gamma,\delta}^{(\xi,\tau)} (v^{\varepsilon})^-$) 
and integrate by parts over 
$Q_\frac{3}{5}^+(\xi,\tau)$ to obtain
$$\int_{Q_\frac{3}{5}^+(\xi,\tau)}y^\gamma(\grad (\zeta^2 F_{\gamma,\delta}^{(\xi,\tau)} (v^{\varepsilon})^+) \grad v^\varepsilon+(\zeta^2 F_{\gamma,\delta}^{(\xi,\tau)} (v^{\varepsilon})^+)\partial_tv^\varepsilon)dxdt
=-\int_{Q'_\frac{3}{5}(\xi,\tau)}(\zeta^2 F_{\gamma,\delta}^{(\xi,\tau)} (v^{\varepsilon})^+)\beta'(u^{\varepsilon})
v^{\varepsilon} dx'dt$$
\begin{equation}
-\int_{Q_\frac{3}{5}^+(\xi,\tau)}\zeta^2 F_{\gamma,\delta}^{(\xi,\tau)} (v^{\varepsilon})^+ f_tdxdt.
\end{equation}
By calculating appropriately and by noticing that due to the non negativity of $\beta'_\varepsilon$ the boundary integral term has the right sign, so it can be omitted, we obtain
$$\int_{Q_\frac{3}{5}^+(\xi,\tau)}y^\gamma(F_{\gamma,\delta}^{(\xi,\tau)}|\grad (\zeta (v^\varepsilon)^+)|^2
+\frac{1}{2}[\grad F_{\gamma,\delta}^{(\xi,\tau)}\grad (\zeta (v^\varepsilon)^+)^2
+F_{\gamma,\delta}^{(\xi,\tau)}\partial_t (\zeta (v^\varepsilon)^+)^2])dxdt$$
 $$\leq \int_{Q_\frac{3}{5}^+(\xi,\tau)} y^\gamma F_{\gamma,\delta}^{(\xi,\tau)}(|\grad \zeta|^2+\zeta \zeta_t)((v^\varepsilon)^+)^2dxdt
+\frac{1}{2}\int_{Q_\frac{3}{5}^+(\xi,\tau)}y^\gamma \grad F_{\gamma,\delta}^{(\xi,\tau)}\grad \zeta^2((v^\varepsilon)^+)^2dxdt$$ 
$$+\int_{Q_\frac{3}{5}^+(\xi,\tau)} \zeta^2 F_{\gamma,\delta}^{(\xi,\tau)}(v^\varepsilon)^+ f_tdxdt.$$  
Using the properties of the $\delta -$smoothing of the fundamental solution and that, for $\delta$ small enough, the inequalities
$0\leq F_{\gamma,\delta}^{(\xi,\tau)}\leq C(n,\gamma)\ \ in\ \   (B_{\frac{4}{5}}^+\setminus B_\frac{1}{5}^+)\times(-\frac{2}{5},0)$,
and  
$c(n,\gamma)\leq F_{\gamma,\delta}^{(\xi,\tau)} \leq C(n,\gamma)\ \ in\ \ B_{\frac{4}{5}}^+\times (-\frac{2}{5},-\frac{4}{25}),$ 
we have
$$\int_{Q_\frac{2}{5}^+(\xi,\tau)} y^\gamma F_{\gamma,\delta}^{(\xi,\tau)}|\grad (v^{\varepsilon})^+|^2dxdt
+ \fint_{Q_{\delta}^+(\xi,\tau)} y^\gamma(v^{\epsilon})^+dxdt\leq C(n,\gamma)  \int_{-\frac{2}{5}}^{-\frac{4}{25}} \int_{B_\frac{4}{5}^+}y^\gamma ((v^\varepsilon)^+)^2dxdt$$
\begin{equation}\label{estimate 1}
+C(n,\gamma)\int_{-\frac{2}{5}}^0 \int_{B_\frac{4}{5}^+\setminus B_\frac{1}{5}^+}y^\gamma ((v^\varepsilon)^+)^2dxdt+ C(n,\gamma)M
\end{equation}
where 
$M:=||v^\varepsilon||_\infty ||f_t||_\infty$.
 
Now, we first let $\varepsilon$ tend to $0$ in order to obtain (\ref{estimate 1}) for $v^+$, then we let $\delta$ to go to $0$, and, finally, we take the supremum over 
$(\xi,\tau)\in Q_\frac{1}{4}^+$ to obtain, a fortiori,
$$\int_{Q_\frac{1}{5}^+}y^\gamma F_\gamma (x,-t)|\grad v^+|^2dxdt
+\sup_{Q_\frac{1}{5}^+}(v^+)^2
\leq C(n,\gamma)\int_{-\frac{2}{5}}^{-\frac{4}{25}}\int_{B_\frac{4}{5}^+}y^\gamma (v^+)^2dxdt$$
\begin{equation}\label{inequality for iteration}
+C(n,\gamma)\sup_{Q_1^+\setminus Q_\frac{1}{5}^+}(v^+)^2+C(n,\gamma)M.
\end{equation}
Next we want to control the first integral of the right hand side of (\ref{inequality for iteration}) by one similar to the first integral of the left hand side of (\ref{inequality for iteration}). To do that we first multiply the equation in (\ref{penalized equation}) by $\zeta^2(v^\varepsilon)^+$ where $\zeta$ is a smooth cutoff function supported in $B_1\times(-1,t)$ , for any $t\leq-\frac{4}{25}$, $\zeta \equiv 1$ on $B_\frac{4}{5} \times (-\frac{2}{5},t)$, and vanishing near its parabolic boundary with $|\grad\zeta|\leq c$ and $0\leq\zeta_t\leq c$, then we integrate by parts over this set intersected by $\R^n_+$ to have
$$\int_{-1}^t\int_{B_1^+}y^\gamma(\grad (\zeta^2(v^{\varepsilon})^+) \grad v^\varepsilon+(\zeta^2
(v^{\varepsilon})^+)\partial_tv^\varepsilon)dxdt
=-\int_{-1}^t\int_{B_1'}(\zeta^2 (v^{\varepsilon})^+)\beta'(u^{\varepsilon})
v^{\varepsilon} dx'dt$$ 
$$-\int_{-1}^t\int_{B_1^+}\zeta^2
(v^{\varepsilon})^+ f_tdxdt.$$
Again, exploiting the positivity of $\beta'$ and letting $\varepsilon$ go to zero, we arrive, as above but in a much simpler way, at the following inequality
$$\int_{B_ \frac{4}{5}^+}y^\gamma (v^+)^2(x,t)dx
+\int_{-\frac{2}{5}}^t\int_{B_\frac{4}{5}^+}y^\gamma |\grad  v^+|^2dxdt
\leq c(n,\gamma)
%\int_{-\frac{-2r^2}{5}}^{-\frac{4r^2}{25}}\int_{B_r^+} 
%(|\grad\zeta|^2+\zeta\zeta_t)(v^+)^2dxdt
\int_{Q_1^+}y^\gamma (v^+)^2dxdt+C(n,\gamma)Mr^{n+2-\gamma}$$
$\forall\ t \in (-\frac{2}{5},-\frac{4}{25})$. Observe that
%$r$ is chosen in such a way that 
a sufficient portion of the coincidence set is present in $Q_1$ so that the parabolic Poincar\'{e} inequality can be applied to dominate the integral on the right hand side of the above inequality. Therefore, since the second term on the left hand side is non negative, we have, for every 
$-\frac{2}{5}\leq t\leq-\frac{4}{25}$,
$$\int_{B_ \frac{4}{5}^+}y^\gamma(v^+)^2(x,t)dx
\leq C(n,\gamma)\int_{Q_1^+}y^\gamma|\grad v^+|^2dxdt
+C(n,\gamma)M.$$
We then integrate the above inequality with respect to $t$ from $-\frac{2}{5}$ to $-\frac{4}{25}$ to get
$$\int_{-\frac{2}{5}}^{-\frac{4}{25}}
\int_{B_ \frac{4}{5}^+}y^\gamma(v^+)^2dxdt
\leq C(n,\gamma)\int_{Q_1^+}y^\gamma|\grad(v^-)|^2dxdt
+C(n,\gamma)M.$$
Insert this in (\ref{inequality for iteration}) above and, using the fact that $F_\gamma(x,-t)\geq c(n,\gamma)$ for 
$-\frac{2}{5}\leq t\leq-\frac{4}{25}$, 
to have  
$$\int_{Q_\frac{1}{5}^+}y^\gamma F_\gamma(x,-t)|\grad v^+|^2dxdt
+\sup_{Q_\frac{1}{5}^+}(v^+)^2\leq$$
\begin{equation}\label{hole iteration}
C(n.\gamma)(\int_{Q_1^+ \setminus Q_\frac{1}{5}^+}
y^\gamma F_\gamma(x,-t)|\grad v^-|^2dxdt
+\sup_{Q_1^+\setminus Q_\frac{1}{5}^+}(v^-)^2))
+C'(n,\gamma)M.
\end{equation}
Set 
$\omega(\rho):=\int_{Q_\rho^+}y^\gamma F_\gamma|\grad v^-|^2dxdt
+\sup_{Q_\rho^+}(v^-)^2$, 
then add 
$C(n,\gamma)\omega(\frac{1}{5})$ 
to both sides of (\ref{hole iteration}) and divide the new inequality by $1+C(n,\gamma)$ to have
\begin{equation}\label{iteration 1}
\omega(\frac{1}{5})\leq \lambda \omega (1)+c
\end{equation}
where $\lambda =\frac{C(n,\gamma)}{1+C(n,\gamma)}$. Iteration of (\ref{iteration 1}) implies that there exists an $\alpha=\alpha(\lambda)\in (0,1)$ and a constant 
$C=C(n,\gamma,||u_t||_{\infty},||f_t||_\infty)$ such that
$$\omega (\rho)\leq C\rho^\alpha$$  
for every $0<\rho\leq\frac{r_0}{5}$. This concludes the H\"{o}lder continuity from the past for $v^+$ and, of course, for $v^-$ since the above estimates are identical. The continuity from the future follows, now, by standard methods. 
\end{proof}

\section{Free boundary regularity}\label{SectionFB}
In this section we study the space-time smoothness of the free boundary. Of course, one expects such a result only around free boundary points of certain property. This property has to guarantee that the speed of the interface is under control. The speed involves time derivative of the solution and \S \ref{SectionTime} provides us with the right framework. As in \cite{ACM1}, we work with  the zero obstacle and with the right hand side of the equation to vanish at the free boundary point of interest, which, for simplicity, we take it to be the origin. Therefore we set $\widetilde{u}(x',y,t)=u(x',y,t)-\psi(x',t)+\frac{1}{4(1-s)}(\Delta\psi(0,0)-\partial_t\psi(0,0))y^2$
and we see that, away from the coincidence (zero) set, $\widetilde{u}$ solves $$\div(|y|^\gamma\grad \widetilde{u})-|y|^\gamma\partial_t \widetilde{u}=|y|^\gamma g(x',t)$$
where $g(x',t)=\Delta \psi(0,0)-\partial_y\psi(0,0)-(\Delta\psi(x,t)-\partial_t\psi(x,t))$. It is clear that  the regularity of $\tilde{u}$ is equivalent to the regularity of $u$ as long as $\psi$ is smooth enough, also $\{\tilde{u}(x',y,t)=0\}=\{u(x',y,t)=\psi(x',t)\}$ and upon reflection $\tilde{u}$ in $B^*_1:=\{(x,t)\in \R^{n+1}:|x|^2+t^2<1\}$ satisfies:
\begin{equation}\label{4.28}
\begin{cases}
\tilde{u}(x',0,t)\geq 0 & {\rm{in}}\ \ B^*_1\cap\{y=0\}\\
\tilde{u}(x',y,t)=\tilde{u}(x',-y,t) & {\rm{in}} \ \ B^*_1\\
L_\g \tilde{u}=|y|^\gamma g(x',t) & {\rm{in}}\ \ B^*_1\setminus\{\tilde{u}=0\}\\
L_\g \tilde{u}\leq |y|^\gamma g(x',t) & {\rm{in}}\ \ B^*_1 \\
\end{cases}
\end{equation}
where  $L_\g u:=\div(|y|^\gamma\grad u)-|y|^\gamma\partial_tu$.  

In the rest of this section, for the sake of simplicity we just\; "drop"\; the\; "$\sim$"\; . Notice that our approach will also apply when we subtract an appropriate $L_\gamma-$heat polynomial to compensate with the case of the general non-homogeneous right hand side. 

Observe, now, that, as it was done in  \cite{ACM1}, if we pass the term involving $u_t$ to the right hand side of the equation and "freeze" time then we can apply the elliptic theory of \cite{CSS} and \cite{CDeSSa}. Thus, if the origin is a "non-degenerate" free boundary point, then in space neighborhood  the free boundary interface is smooth; moreover, its blow-up limit, unique up to rotations, is given (in appropriate coordinates) by
$$u_0:=\frac{1}{1-s^2}\rho^{1+s}(2\cos^{2(s+1)}\frac{\theta}{2•}-(1+s)cos^{2s}\frac{\theta}{2•})$$
where $\rho:=\sqrt{x_1^2+y^2}$ and $\theta:=\arctan (\frac{y}{x_1})$. Notice that when $s=\frac{1}{2}$ the above is reduced to 
$\frac{2}{3}\rho^{\frac{3}{2}}\cos(\frac{3}{2}\theta)$ via the identity $\cos(\frac{3}{2}\theta)=(4\cos^2\frac{\theta}{2}-3)cos\frac{\theta}{2}$.

Indeed, these results are very useful to us thanks to the quasi-convexity property. Consequently, in order to preserve the geometry of the interface at a point the "scaling" must be "hyperbolic", and thus a "hyperbolic" definition of non-degeneracy is in order:
 
\begin{defn}\label{nondeg-point}
Let $(x_0,t_0)$ be a free boundary point and 
$B_r^*(x_0,t_0):=\{(x,t)\in R^{n+1}: (x-x_0)^2+(t-t_0)^2<r^2\}$, set $$l:=\limsup_{r\to 0^+}\frac{||u||_{L^\infty(B_r^*(x_0,t_0))}}{r^{1+s}}$$
A point $(x_0,t_0)$ is called a non-degenerate free boundary point if it is of positive parabolic density of the coincidence set and $0<l<\infty$, otherwise degenerate.
\end{defn}

%To study the regularity of non-degenerate free boundary points, we follow the ideas of subsection 4.3 in \cite{ACM1}. The next theorem provides the smoothness of the free boundary around non-degenerate points.
Now, we state the main theorem of this section:

\begin{thm}\label{FreeBdry}
Let $u$ be a solution to (\ref{FBP3}). Assume the origin to be a non-degenerate free boundary point. Then the free boundary is a $C^{1,\alpha}$ $n$-dimensional surface about the origin.
\end{thm}

To prove Theorem \ref{FreeBdry}, we must identify first, as it was done in \cite{ACM1}, the  global profiles. This is the context of the following lemma, where we consider limits of "hyperbolic" blow up sequences 
%of the form %$$u_r(x,t):=\frac{u(rx',ry,rt)}{r^{1+s}}.$$
  i.e.. limits of $u_r(x',y,t):=\frac{u(rx',ry,rt)}{r^{1+s}}$ as $r\to 0$.

\begin{lemma}\label{blowups}
Let $u$ be a solution to (\ref{FBP3}). If $(0,0)$ is a non-degenerate free boundary point then there exists a sequence  $u_{r_j}$ of blow ups which converges uniformly on compact subsets to a function $u_0$ such that, (in appropriate coordinates),  
\begin{equation}\label{ellipticprofile}
u_0(x',y,t):=\frac{1}{1-s^2}\rho(t)^{1+s}(2\cos^{2(s+1)}\frac{\theta(t)}{2•}-(1+s)cos^{2s}\frac{\theta(t)}{2•})
%u_0(x,y,t)=C\bigg(\sqrt{(x\cdot e+\omega t)^2+y^2}+(x\cdot %e+\omega t) \bigg)^s\bigg((x\cdot e+\omega t)-s\sqrt{(x\cdot e+\omega t)^2+y^2}\bigg)
%\partial_eu_0=C\left(\sqrt{(x'\cdot e)^2+y^2}-(x'\cdot e)\right)^s
\end{equation}
where $\rho(t):=\sqrt{(x_1+\omega t)^2+y^2}$ and $\theta(t):=\arctan(\frac{y}{x_1+\o t)})$.
%for some tangential space direction $e$, $\omega\in \R$ and $C$ is a constant depending on $s$.  In particular, for any fixed $t-$level fixed, $u_0(x,0,t)$ coincide with the corresponding elliptic profile. 
\end{lemma}
\begin{proof}
Since  $(0,0)$ is assumed to be a non-degenerate free boundary point we have that $0<l<\infty$, therefore we can extract a subsequence $u_{r_j}$ converging uniformly on compact subsets to a non trivial limit, which we denote by  $u_0$. We will prove that $u_0$ corresponds to the elliptic profile and as such it should satisfy (\ref{ellipticprofile}).

This $u_0$ is a $L_\gamma$ harmonic function for every fixed $t$ outside of the coincidence set; and the coincidence set, due to the density assumption, is a convex cone in $\R^n$, or more precisely in 
$(x',t)$ 
variables. Using elliptic Almgren's frequency formula (see \cite{CDeSSa} (Theorem 6.2) or the Appendix of \cite{ACM1} where the right hand side is allowed to be only in $L^p$) and the elliptic theory developed in \cite{CSS}, we conclude that, at $t=0$ 
$u_0$ is the corresponding elliptic profile. Moreover the convex cone is composed by the following two supporting hyperplanes $Ax_1+at=0$ for $t\geq0$ and $Bx_1+bt=0$ for $t\leq0$ with the constants $A\geq0$, $B\geq0$ and $bA\leq{aB}$. 
We want to prove that this convex cone is actually a non-horizontal half space i.e. $A>0$, $B>0$, and $bA=aB$, and $u_0$ admits the stated representation; The argument is divided into 4 steps.

\textit{Step I: $A>0$ and $B>0$.}

For, if $A=0$ then for every $t>0$ $u_0(x,t)$ is $L_\gamma$ harmonic in all of $\R^n$. But for $t=0$ $u_0$ has certain degree of growth, therefore, by continuity of $u_0$, a contradiction. Similarly $B>0$.

\textit{Step II: For each fixed t, $u \sim |x|^{1+s}$ as $|x| \to \infty$ with $x\cdot e_1\geq \varepsilon$ for some $\varepsilon>0$.}

It is enough to show the bound by below. Therefore take a sequence  $x^{(j)}$ such that $|x^{(j)}| \to \infty$ with $x^{(j)} \cdot e_1 \geq \varepsilon$ for every $j$ then by convexity $u_0(x^{(j)},t) \geq u_0(x^{(j)},0) + (u_0)_t(x^{(j)},0)t$, hence by the  behavior of $u_0$ at $t=0$ the result follows.

\textit{Step III: Representation for each fixed $t$.}  We claim that, for each fixed $t$, 
$$u_0(x',y,t):=\frac{1}{1-s^2}\rho(t)^{1+s}(2\cos^{2(s+1)}\frac{\theta(t)}{2•}-(1+s)cos^{2s}\frac{\theta(t)}{2•})$$
where for 
$t>0$, 
$$\rho(t)=\sqrt{(x_1+\frac{a}{A}t)^2+y^2}, \ \ \text{and} \ \ \theta(t) = \arctan\frac{y}{x_1+\frac{a}{A}t}$$
and for 
$t<0$, 
$$\rho(t)=\sqrt{(x_1+\frac{b}{B}t)^2+y^2}, \ \ \text{and} \ \ \theta (t) = \arctan \frac{y}{x_1+\frac{b}{B}t}.$$

Indeed, for each fixed 
$t>0$, $u_0$ 
is an $L_\gamma$  harmonic function which vanishes for 
$\{x_1\leq -\frac{a}{A}t\}\cap\{x_n=0\}$
and grows at infinity with 
$1+s$ 
exponent, therefore by Phragmen-Lindelof theorem we obtain the representation. Analogously, for $t<0$.  

\textit{Step IV: $bA=aB$.}

For, if not then 
$$\partial_t u_0(0,0,0^+)-\partial_t u_0(0,0,0^-)
%(\frac{a}{A}-\frac{b}{B})\rho^\frac{1}{2}\cos(\frac{1}{2}\theta)
\neq0,$$ 
whence, by approximation, a contradiction to the continuity of $\partial_t u$ at the origin.

Set $\omega:=\frac{a}{A}$ and the proof is complete.
\end{proof}

Finally we prove our theorem:
\begin{proof}
Obviously the existence of $\omega$ in Lemma \ref{blowups} implies the differentiability of the free boundary at the origin. Also, due to the upper semi-continuity of the elliptic Almgren's frequency function, we have the differentiability of the free boundary for any nearby point $p=(x_p,t_p)$ at least when $t_p\leq 0$, since $u_t$ is continuous there. Now, if $t_p>0$ and $p=(x_p,t_p)$ still near the origin, we observe that the frequency function will converge to $1+s$, and this implies that the positive density will propagate to $p$. Consequently, the point $p=(x_p,t_p)$ will be a free boundary point of positive parabolic density with respect to zero set, which renders $u_t$ continuous there. Hence we have the differentiability of the free boundary there, too. To prove the continuous differentiability of it consider two distinct free boundary points nearby, say p and 0. Assume, on the contrary, that it is not true, that is $\omega(p)$ does not converge to $\omega(0)$ as $p \to 0$. Consider the blow up sequences  $u_{r_j}^{(p)}$ and $u_{r_i}^{(0)}$ around $p$ and $0$, respectively, where $u_{r_j}^{(p)}(x,t):=\frac{u(r_j((x,t)-p))}{r_J^{1+s}}$. 
These sequences converge uniformly to 
the corresponding elliptic profiles.
%$$u_0^{(p)}(x,t):= \frac{2}{3}\rho^\frac{3}{2}(p,t)\cos \frac{3}{2}\theta (p,t)$$
%and 
%$$u_0^{(0)}(x,t):= \frac{2}{3}\rho^\frac{3}{2}(0,t)\cos \frac{3}{2}\theta (0,t)$$
%respectively, where 
%$\rho(p,t):=\sqrt{(x_1(p)+\omega(p)t(p))^2+x_n^2)}$ and $\theta(p,t):=\arctan\frac{x_n}{x(p))+\omega(p)t(p)}$. 

So, if $\omega(p)$ does not converge to $\omega(0)$ then $u_0^{(p)}$ does not converge to $u_0^{(0)}$, therefore a contradiction to the continuity of the solution $u$. Hence a $C^{\alpha}$ estimate of the free boundary normals follows easily. 
\end{proof}

%
%\begin{thm}\label{blowup}
%There exists a unique (up to rotations and multiplicative constants) non-zero global solution $u$ of homogeneity $1+s$. 
%%In addition $$\partial_eu=C\left(\sqrt{(x'\cdot e)^2+y^2}-(x'\cdot e)\right)^s$$ for some tangential space direction $e$. 
%\end{thm}
%
%\begin{proof}
%Let $u$ be a $k-$homogeneous solution.
%
%\textit{Claim 1.} $\partial_e u$ is either nonnegative or nonpositive for any tangential direction $e$. That is $u$ depends only on one tangential direction and is monotone along that direction.
%
%\textit{Claim 2.} $(\partial_e u)_+$ and  $(\partial_e u)_-$ are complementary halfspaces at almost all $t$.
%
%\textit{Claim 3.} $k=1+s$ or $u\equiv0$. 
%
%\end{proof}

\section*{Acknowledgments} 
%\addcontentsline{toc}{section}{Acknowledgments}
L. Caffarelli was supported by NSF grants. E. Milakis was supported by Marie Curie International Reintegration Grant No 256481 within the 7th European Community Framework Programme and the University of Cyprus research grants. Part of this work was carried out while the first and the third authors were visiting the University of Texas. They wish to thank the Department of Mathematics and the Institute for Computational Engineering and Sciences for the warm hospitality and support

%\section*{Frequency Formula} \label{frequency}

\bibliographystyle{plain}   % Here the bibliography
\bibliography{biblio}             % is inserted.
\index{Bibliography@\emph{Bibliography}}%

\vspace{2em}

\begin{tabular}{l}
Ioannis Athanasopoulos\\ University of Crete \\ Department of Mathematics  \\ 71409 \\
Heraklion, Crete GREECE
\\ {\small \tt athan@uoc.gr}
\end{tabular}
\begin{tabular}{l}
Luis Caffarelli\\ University of Texas \\ Department of Mathematics  \\ TX 78712\\
Austin, USA
\\ {\small \tt caffarel@math.utexas.edu}
\end{tabular}
\begin{tabular}{lr}
Emmanouil Milakis\\ University of Cyprus \\ Department of Mathematics \& Statistics \\ P.O. Box 20537\\
Nicosia, CY- 1678 CYPRUS
\\ {\small \tt emilakis@ucy.ac.cy}
\end{tabular}

\end{document}